\documentclass[12pt,twoside]{amsart}
\usepackage{amssymb,color,tikz} \usetikzlibrary{matrix}
\nonstopmode \textwidth=16.00cm \textheight=24.00cm \topmargin=-1.00cm
\numberwithin{equation}{section} 
\newcommand {\PP}{\mathbb{P}}
\DeclareMathOperator{\codim}{codim} 
\DeclareMathOperator{\Soc}{soc} \def\cocoa{{\hbox{\rm C\kern-.13em
      o\kern-.07em C\kern-.13em o\kern-.15em A}}}
\newtheorem{theorem}{Theorem}[section]
\newtheorem{lemma}[theorem]{Lemma}
\newtheorem{proposition}[theorem]{Proposition}
\newtheorem{corollary}[theorem]{Corollary}
\newtheorem{conjecture}[theorem]{Conjecture} \theoremstyle{definition}
\newtheorem{definition}[theorem]{Definition} \theoremstyle{remark}
\newtheorem{remark}[theorem]{Remark}
\newtheorem{example}[theorem]{Example}
\newtheorem{notation}[theorem]{Notation}
\definecolor{MyDarkGreen}{cmyk}{0.7,0,1,0}

\begin{document}
\title[The non-Lefschetz locus]%
{The non-Lefschetz locus}

\author[M.\ Boij]{Mats Boij} \address{Department of Mathematics, KTH
  Royal Institute of Technology, S-100 44 Stockholm, Sweden}
\email{boij@kth.se}

\author[J.\ Migliore]{Juan Migliore} \address{ Department of
  Mathematics, University of Notre Dame, Notre Dame, IN 46556, USA}
\email{migliore.1@nd.edu}

\author[R.\ M.\ Mir\'o-Roig]{Rosa M.\ Mir\'o-Roig} \address{Facultat
  de Matem\`atiques, Departament de Matem\`atiques i Inform\`atica,
  Gran Via des les Corts Catalanes 585, 08007 Barcelona, Spain}
\email{miro@ub.edu}

\author[U.\ Nagel]{Uwe Nagel} \address{Department of Mathematics,
  University of Kentucky, 715 Patterson Office Tower, Lexington, KY
  40506-0027, USA} \email{uwe.nagel@uky.edu}

\pagestyle{plain}
\begin{abstract}
  We study the weak Lefschetz property of artinian Gorenstein algebras
  and in particular of artinian complete intersections.  In
  codimension four and higher, it is an open problem whether all
  complete intersections have the weak Lefschetz property.
  
  For a given artinian Gorenstein algebra $A$ we ask what linear forms
  are Lefschetz elements for this particular algebra, i.e., which
  linear forms $\ell$ give maximal rank for all the multiplication
  maps $\times \ell: [A]_i \longrightarrow [A]_{i+1}$. This is a
  Zariski open set and its complement is the \emph{non-Lefschetz
    locus}.

  For monomial complete intersections, we completely describe the
  non-Lefschetz locus. For general complete intersections of
  codimension three and four we prove that the non-Lefschetz locus has
  the expected codimension, which in particular means that it is empty
  in a large family of examples. For general Gorenstein algebras of
  codimension three with a given Hilbert function, we prove that the
  non-Lefschetz locus has the expected codimension if the first
  difference of the Hilbert function is of decreasing type. For
  completeness we also give a full description of the non-Lefschetz
  locus for artinian algebras of codimension two.
\end{abstract}
\maketitle
\section{Introduction}

If $A = R/I$ is an artinian standard graded algebra over the
polynomial ring $R = k[x_1,\dots,x_n]$, where $k$ is a field, then $A$
is said to have the {\em Weak Lefschetz Property (WLP)} if the
homomorphism induced by multiplication by a general linear form, from
every degree to the next, has maximal rank.  In this paper we will
always assume that $k$ has characteristic zero.

A famous result in commutative algebra says that an artinian monomial
complete intersection over a field of characteristic zero has the WLP
(and even a stronger condition called the Strong Lefschetz
Property). This was proved in \cite{stanley}, \cite{watanabe}, and
\cite{RRR}. A consequence of this is that if the generator degrees are
specified, a general complete intersection with those generator
degrees has the WLP. It is an open question whether {\em every}
complete intersection has the WLP. Notice that the result above fails
to distinguish between a monomial complete intersection and a general
one (always with fixed generator degrees). We give a finer measure of
the Lefschetz property that does distinguish between these
(conjecturally in all cases,  and we give a proof in $\leq 4$
variables).

Suppose that such a standard graded algebra $A$ is given.  For any
pair of consecutive components $A_i$ and $A_{i+1}$, we can consider
the locus $\mathcal L_i$ of linear forms that fail to induce a
homomorphism of maximal rank on these components.  We will observe
that for each $i$ the variety $\mathcal L_i$ is a determinantal
variety, so depending on the absolute value of the {\em difference}
$\dim [A]_{i+1} - \dim [A]_i$, there is an expected codimension. If
the variety achieves this codimension, its degree (as a possibly
non-reduced scheme) is also known. One can then ask further questions
about $\mathcal L_i$, such as what are its irreducible components.  If
$\mathcal L_i$ fails to have the expected codimension, it is still
determinantal but its degree is less clear.

We define the {\em non-Lefschetz locus} $\mathcal L_I$ to be the union
of these loci $\mathcal L_i$, viewed as subvarieties of the
corresponding projective space $(\mathbb P^{n-1})^*$, over all
possible sets of consecutive components.  The algebra $A$ fails to
have the WLP if and only if $\mathcal L_I = \mathbb (P^{n-1})^*$. The
variety $\mathcal L_I$ is thus a union of determinantal varieties in
general.  If $A$ is Gorenstein (e.g. a complete intersection), there
is a natural sequence of inclusions of the $\mathcal L_i$, so
$\mathcal L_I$ is in fact itself a determinantal variety. (See
Proposition~\ref{inclusions}.)

In this paper we will study the non-Lefschetz locus for {\em specific}
algebras (monomial algebras) and we will consider it in the case of
the general element of an irreducible family (complete intersections
of prescribed generator degrees). Much more difficult is the question
of whether {\em every} element of an irreducible family (specifically
complete intersections) has the WLP, i.e. whether the non-Lefschetz
locus is always of positive codimension for such algebras.

In Section~\ref{sec:monomial} we completely characterize the
non-Lefschetz locus of monomial complete intersections
(Proposition~\ref{monomial ci}) and we also find all the possible
Jordan types of linear forms in such algebras
(Proposition~\ref{prop:Jordan}).

In Section~\ref{sec:general ci} we conjecture that the non-Lefschetz
locus of a general complete intersection has the expected codimenion
in the sense that will be made precise in Section~\ref{Sec:Prel}. We
prove this conjecture for complete intersections of codimension three
(Theorem~\ref{LI,n=3,d1d2d3}) and codimension four
(Theorem~\ref{LI,n=4,d1d2d3d4}).

In Section~\ref{sec:general Gorenstein} we study the non-Lefschetz
locus of a general artinian Gorenstein algebra of codimension three
with a given Hilbert function. In Theorem~\ref{gormain} we prove that
the non-Lefschetz locus has the expected codimension if the $g$-vector
associated to the Hilbert function is of decreasing type, while it is
of codimension one otherwise.

In Section~\ref{sec:codim2} we give a complete description of the
situation for algebras in codimension two.

\subsection*{Acknowledgements} 
  Part of this work was carried out during a visit to IHP in Paris and
  part of the work was carried out during a visit to CIRM in
  Trento. We are very grateful for these opportunities. 
  We are also grateful to Anthony Iarrobino  and Junzo Watanabe for useful
  discussions related to this work.  

  The first author was partially supported by the grant VR2013-4545,
  the second author by the National Security Agency under Grant
  H98230-12-1-0204 and by the Simons Foundation under Grant \#309556,
  the third author by the grant MTM2013-45075-P and the fourth author
  by the National Security Agency under Grant H98230-12-1-0247 and by
  the Simons Foundation under Grant \#317096.  

\section{Preliminaries}\label{Sec:Prel}

Let $R = k[x_1,x_2,\dots ,x_n]$ where $k$ is an algebraically closed
field of characteristic zero.  Let $M$ be a graded $R$-module of
finite length. We first briefly recall an idea, originally due to Joe
Harris, dealing with an isomorphism invariant of $M$. For further
details see \cite{Migliore(GeomInv)}.

The module structure of $M$ is determined by a collection of
homomorphisms $\phi_i : [R]_1 \rightarrow \hbox{Hom}_k(M_i, M_{i+1})$
as $i$ ranges from the initial degree of $M$ to the penultimate degree
where $M$ is not zero. Since $\phi_i$ is trivial if either $[M]_i$ or
$[M]_{i+1}$ is zero, we assume that this is not the case (we do not
assume that $M$ is generated in the first degree, so a zero component
could lie between non-zero ones). Let
$\ell = a_1 x_1 + \dots + a_n x_n$, and let us refer to the $a_i$ as
the {\em dual variables}. If we choose bases for $[M]_i$ and for
$[M]_{i+1}$, we can view $\phi_i$ as a
$(\dim [M]_{i+1}) \times (\dim [M]_i)$ matrix $B_i$ whose entries are
linear forms in the dual variables.  For any fixed $t$ we can thus
consider the ideal of $(t+1) \times (t+1)$ minors of $B_i$, and this
is an isomorphism invariant of $M$. However, for our purposes it is
enough to consider the ideal of maximal minors of $B_i$. Denoting by
$Y_i$ the scheme defined by the ideal of maximal minors of $B_i$, we
can view $Y_i$ as lying in the dual projective space
$(\mathbb P^{n-1})^* = \hbox{Proj}(k[a_1,\dots,a_n])$.  We have an
expected codimension for $Y_i$, and if that codimension is achieved
then we also have a formula for $\deg Y_i$:

\begin{lemma} \label{expected} Without loss of generality assume that
  $\dim [M]_i \leq \dim [M]_{i+1}$ (otherwise consider the transpose
  of $B_i$).  For sufficiently general entries of $B_i$, the
  codimension of $Y_i$ is $\dim [M]_{i+1} - \dim [M]_i +1$. If this
  codimension is achieved, then
  $\deg Y_i = \binom{\dim M_{i+1}}{\dim M_i -1}$.
\end{lemma}

\begin{example}[\cite{Migliore(GeomInv)}]
  Harris's motivation was to apply this machinery to liaison
  theory. For instance, let $C \subset \mathbb P^3$ be the union of
  four general lines. Let
  \[
  M(C) = \bigoplus_{t \in \mathbb Z} H^1(\mathbb P^3, \mathcal
  I_C(t)),
  \]
  the Hartshorne-Rao module of $C$. We have
  \[
  \dim M(C)_t = \left \{
    \begin{array}{cl}
      3 & \hbox{if } t=0; \\
      4 & \hbox{if } t=1; \\
      2 & \hbox{if } t = 2; \\
      0 & \hbox{otherwise}.
    \end{array}
  \right.
  \]
  Taking $M = M(C)$, the expected codimension of $Y_0$ is $4-3+1 =2$,
  and the expected degree is $\binom{4}{2} = 6$. One can show that in
  fact $Y_0$ is the curve in $(\mathbb P^3)^*$ obtained as the duals
  of the four components of $C$ together with the duals of the two
  4-secant lines of $C$. It then follows from the fact that $Y_0$ is
  an isomorphism invariant, and some now-classical results of liaison
  theory (with a small argument), that $C$ is the only union of skew
  lines in its even liaison class.
\end{example}

Our idea now is to apply this machinery to the study of the Weak
Lefschetz property. Traditionally, we say that an artinian algebra
$A=R/I$ has the {\em Weak Lefschetz property (WLP)} if there is a
linear form $\ell \in [A]_1$ such that, for all integers $i$, the
multiplication map
\[
\times \ell: [A]_{i} \to [A]_{i+1}
\]
has maximal rank, i.e.\ it is injective or surjective. In this case,
the linear form $\ell $ is called a {\em Lefschetz element} of $A$.
(We will often abuse terminology and say that the corresponding ideal
has the WLP.)  The Lefschetz elements of $A$ form a Zariski open,
possibly empty, subset of $[A]_1$, which as above we will projectivize
and view in $(\mathbb P^{n-1})^*$.  This open set is nothing but
$(\PP^{n-1})^*\setminus L_I$.  This is our primary focus in this
paper, but we note that $A$ is said to have the {\em Strong Lefschetz
  property (SLP)} if the analogous statements are true for the
multiplication maps
\[
\times \ell^d : [A]_i \to [A]_{i+d}
\]
for all $i$ and $d$.

If we consider $A$ as an $R$-module, to say that $A$ satisfies the WLP
is equivalent to saying that none of the varieties $Y_i$ is all of
$(\mathbb P^{n-1})^*$. We first relabel the $Y_i$ with a more
descriptive notation for our application.

\begin{definition} \label{Lef loc} Given an artinian graded algebra
  $A = R/I$, we define
  \[
  \mathcal L_I:=\{[\ell] \in \mathbb P ([A]_1) \mid \ell \text{ is not
    a Lefschetz element} \} \subset (\PP^{n-1})^\ast
  \]
  and we call it the {\em non-Lefschetz locus of} $I$ (or of $A$).
  For any integer $i\ge 0$, we define
  \[
  \mathcal L_{I,i}:=\{\ell \in [A]_1 \mid \times \ell\colon [A]_i
  \longrightarrow [A]_{i+1} \text{ does not have maximal rank} \}
  \subset (\PP^{n-1})^\ast.
  \]
\end{definition}

In order to study the non-Lefschetz locus from a scheme-theoretic
perspective, we view $\mathcal L_{I,i}$ not as a set but rather as the
subscheme of $(\mathbb P^{n-1})^*$ defined by the maximal minors of a
suitable matrix, as explained above, taking $M = A$. The size of this
matrix is determined by the Hilbert function of $A$.  More precisely,
we introduce $S = k[a_1,a_2,\dots,a_n]$ as the homogeneous coordinate
ring of the dual projective space $(\mathbb P^{n-1})^\ast$, where we
think of the coordinates $a_1,a_2,\dots,a_n$ as the coefficients in
$\ell = a_1x_1+a_2x_2+\cdots+a_nx_n$. For each degree $i$, the
multiplication by $\ell$ on $S\otimes_k A$ gives the map
\[
\times \ell\colon S\otimes_k [A]_i \longrightarrow S\otimes_k
[A]_{i+1}
\]
of free $S$-modules which is represented by a matrix of linear forms
in $S$ given a choice of bases for $[A]_i$ and $[A]_{i+1}$. The locus
$\mathcal L_{I,i}\subseteq (\PP^{n-1})^\ast$ is scheme-theoretically
defined by the ideal of maximal minors of this matrix and we denote
this ideal by $I(\mathcal L_{I,i})$. Observe that this ideal is
independent of the choice of bases. In this way, we have
$\mathcal L_I = \bigcup_{i\ge 0} \mathcal L_{I,i}$, and
$\mathcal L_I\subseteq (\mathbb P^{n-1})^\ast$ is defined by the
homogeneous ideal
$I(\mathcal L_I) = \bigcap_{i\ge 0} I(\mathcal L_{I,i})$.

\begin{definition}
  If $\hbox{codim } \mathcal L_{I,i}$ takes the value prescribed by
  Lemma \ref{expected}, where now $\dim [M]_i$ is the value of the
  Hilbert function of $A$ in degree $i$, (and hence the degree of
  $\mathcal L_{I,i}$ is also determined by the Hilbert function), then
  we say that $\mathcal L_i$ has the {\em expected codimension} and the
  {\em expected degree}.
\end{definition}

Since in this article we are studying Gorenstein algebras, especially
complete intersections, it will be useful to know that the
non-Lefschetz locus is determined by the failure of injectivity of the
multiplication by linear forms in a single degree. It is clear on a
set-theoretical level that this is true (cf. \cite[Proposition
2.1]{MMN}).  We will now look at the question when there is an
inclusion of the ideals
$I(\mathcal L_{I,i+1})\subseteq I(\mathcal L_{I,i})$ which will ensure
that we only have to consider the middle degree even when we look at
the non-Lefschetz locus defined scheme-theoretically and not only
set-theoretically.

\begin{proposition}\label{inclusions}
  If $h_A(i)\le h_A(i+1) \le h_A(i+2)$ and $[\Soc{A}]_i=0$, then
  $I(\mathcal L_{I,i+1})\subseteq I(\mathcal L_{I,i})$.
\end{proposition}

\begin{proof}
  The ideal $I(\mathcal L_{I,i+1})$ is generated by the maximal minors
  of the matrix representing the map
  $\times \ell \colon S\otimes_k [A]_{i+1} \longrightarrow S\otimes_k
  [A]_{i+2}$,
  where $\ell = a_1x_1+a_2x_2+\cdots+a_nx_n$. Each such minor equals
  the determinant of the matrix representing the map
  \[
  \times \ell \colon S\otimes_k [B]_{i+1} \longrightarrow S\otimes_k
  [B]_{i+2}
  \]
  where $B=A/J$ and $J$ is an ideal generated by $h_A(i+2)-h_A(i+1)$
  forms of degree $i+2$. Since $[A]_i = [B]_i$ and
  $[A]_{i+1} = [B]_{i+1}$, we can prove the inclusion
  $I(\mathcal L_{I,i+1})\subseteq I(\mathcal L_{I,i})$ for $A$ by
  proving the inclusion for all such quotients $B = A/J$. Therefore,
  we will now assume that $h_A(i+1)=h_A(i+2)$.

  Suppose that $\mathcal L_{I,i+1} = (\mathbb P^{n-1})^*$.  Then
  $I(\mathcal L_{I,i+1})=\langle 0\rangle$ and the inclusion of ideals
  is trivial. If $\mathcal L_{I,i}= (\mathbb P^{n-1})^*$ we will also
  have that $\mathcal L_{I,i+1} = (\mathbb P^{n-1})^*$ since $A$ by
  assumption does not have socle in degree $i$ and the inclusion of
  ideals is again trivial. Thus we only have to consider the case when
  $\mathcal L_{I,i} \ne (\mathbb P^{n-1})^*$ and
  $\mathcal L_{I,i+1} \ne (\mathbb P^{n-1})^*$. In this case, we can
  change coordinates so that
  $\times x_n \colon [A]_i\longrightarrow [A]_{i+1}$ and
  $\times x_n \colon [A]_{i+1}\longrightarrow [A]_{i+2}$ both have
  maximal rank. Consider the diagram
  \begin{center}
    \begin{tikzpicture}
      \matrix (m) [matrix of math nodes,row sep=3em,column
      sep=4em,minimum width=2em]
      {S\otimes_k{[A]}_{i} & S\otimes_k{[A]}_{i+1} \\
        S\otimes_k{[A]}_{i+1} & S\otimes_k {[A]}_{i+2} \\};
      \path[-stealth] (m-1-1) edge node [left] {$\times x_n$} (m-2-1)
      edge node [above] {$\times \ell$} (m-1-2) (m-2-1.east|-m-2-2)
      edge node [above] {$\times \ell$} (m-2-2) (m-1-2) edge [double]
      node [right] {$\times x_n$} (m-2-2);
    \end{tikzpicture}
  \end{center}
  The injectivity of the two vertical maps shows that we can choose
  monomial cobases for $[A]_i$, $[A]_{i+1}$ and $[A]_{i+2}$ in such a
  way that the matrix representing the map
  $\times \ell \colon S\otimes_k [A]_i\longrightarrow
  S\otimes_k[A]_{i+1}$
  is a submatrix of the matrix representing the map
  $\times\ell \colon S\otimes_k[A]_{i+1}\longrightarrow
  S\otimes_k[A]_{i+2}$.
  The ideal $I(\mathcal L_{I,i+1})$ is principal, generated by the
  determinant of the matrix representing the map
  $\times\ell \colon S\otimes_k[A]_{i+1}\longrightarrow
  S\otimes_k[A]_{i+2}$.
  Since the two matrices have the same number of rows, the Laplace
  expansion of the determinant of the larger matrix shows that this
  determinant is in the ideal generated by the maximal minors of the
  submatrix, which proves the inclusion
  $I(\mathcal L_{I,i+1})\subseteq I(\mathcal L_{I,i})$.
\end{proof}

\begin{corollary}\label{Scheme-theoretic}
  If $A=R/I$ is Gorenstein of socle degree $e$ then
  $\mathcal L_I = \mathcal L_{I,i}$ scheme-theoretically, where
  $i = \lfloor \frac{e-1}2\rfloor$.
\end{corollary}

\begin{proof}
  If $A$ does not have the WLP, we have
  $\mathcal L_I = \mathcal L_{I,i}=(\PP^{n-1})^\ast$. If $A$ has the
  WLP the Hilbert function is unimodal and by
  Proposition~\ref{inclusions} and the duality of the Gorenstein
  algebra we get the equality.
\end{proof}

\begin{remark}
  If $A=R/I$ has socle in degree $i$, we need not have the inclusion
  $I(\mathcal L_{I,i+1})\subseteq I(\mathcal L_{I,i})$ since we then
  have that $I(\mathcal L_{I,i}) = \langle 0\rangle$ while
  $I(\mathcal L_{I,i+1})$ might be non-trivial.

  If $A$ is not Gorenstein, but \emph{level}, we can get a similar
  result as Corollary~\ref{Scheme-theoretic} but in some cases we will
  have to use two degrees instead of one since we cannot apply
  duality.  (cf. \cite[Proposition 2.1]{MMN} for the set-theoretic
  statement.)
\end{remark}

\section{The non-Lefschetz locus of a monomial complete
  intersection}\label{sec:monomial}

In this section and the next we will restrict ourselves to the case of
complete intersections.  In this section we study monomial complete
intersections.

Notice that to say that an artinian ideal $I\subset R$ has the WLP is
equivalent to saying that $\codim \mathcal L_I \ge 1$. The aim of this
section is to study $\codim \mathcal L_I $ when
$I = \langle F_1, \cdots , F_n \rangle \subset R$ is a monomial
complete intersection.  We know that for a monomial complete
intersection, hence for a general choice of $F_1,\cdots ,F_n$, $R/I$
has the WLP, thanks to the main result of \cite{stanley},
\cite{watanabe} and \cite{RRR}; and the same holds for {\em any}
choice of $F_i$ if $n\le 3$ (cf. \cite{HMNW}).  Nevertheless, we will
see in this section and the next that the non-Lefschetz locus behaves
very differently for monomial complete intersections than it does for
general complete intersections.

\begin{proposition} \label{monomial ci} Let
  $I=\langle x_1^{d_1}, \cdots , x_n^{d_n} \rangle \subset R :=
  k[x_1,\cdots, x_n]$
  be an artinian monomial complete intersection, with socle degree
  $e = d_1 + \dots + d_n - n$.  Assume without loss of generality that
  $d_n \geq \dots \geq d_1 \geq 2$. Then the following
  characterization of the Lefschetz elements holds.
  \begin{enumerate}
  \item\label{MonCI1} If $d_n > \lfloor \frac{e+1}{2} \rfloor$ then
    $\ell = a_1x_1 + a_2 x_2 + \cdots + a_{n} x_{n}$ is a Lefschetz
    element if and only if $a_{n} \neq 0$.
  \item\label{MonCI2} If $e$ is even and
    $d_n\le\lfloor \frac{e+1}{2} \rfloor$ then
    $\ell = a_1x_1 + a_2 x_2 + \cdots + a_{n} x_{n}$ is a Lefschetz
    element if and only if $a_i=0$ for at most one index $i$ and
    $a_j\ne 0$ for all indices $j$ with $d_j > 2$.
  \item\label{MonCI3} If $e$ is odd and
    $d_n\le\lfloor \frac{e+1}{2} \rfloor$ then
    $\ell = a_1x_1+a_2 x_2 + \dots + a_{n} x_{n}$ is a Lefschetz
    element if and only if $a_1a_2\cdots a_n\neq 0$.
  \end{enumerate}
\end{proposition}

\begin{proof}
  We start by fixing the linear form
  $\ell = a_1x_1+a_2x_2+\cdots+a_nx_n$. Let $A= R/I$, let
  $S = \{i\colon a_i\ne 0\}\subseteq \{1,2,\dots,n\}$ and define the
  subrings $A'$ and $A''$ of $A$ as the subrings generated by
  $\{x_i\}_{i\in S}$ and by $\{x_i\}_{i\notin S}$, respectively. Both
  $A'$ and $A''$ are monomial complete intersections and $\ell$ acts
  trivially on $A''$ while it is a Lefschetz element on $A'$.

  In order to determine whether or not $\ell$ is a Lefschetz element
  on $A$, it is sufficient to consider the injectivity of the
  multiplication map in the middle degree, i.e.,
  \begin{equation}
    \times \ell \colon [A]_{\lfloor \frac{e-1}2\rfloor}\longrightarrow
    [A]_{\lfloor \frac{e+1}2\rfloor}.\label{middle_map}
  \end{equation}
  For any integer $j$, we have that
  \[[A]_j =\left( [A']_j\otimes [A'']_0\right)\oplus \left([A']_{j-1}\otimes [A'']_1\right)\oplus
  \cdots \left(\oplus [A']_0\otimes [A'']_j\right)\]
  and since $\ell$ acts trivially on $A''$, the injectivity of
  (\ref{middle_map}) is equivalent to injectivity in each component
  \[
  \times \ell \colon [A']_{\lfloor \frac{e-1}2\rfloor-i}\otimes
  [A'']_i\longrightarrow [A']_{\lfloor \frac{e+1}2\rfloor-i}\otimes
  [A'']_i,\qquad \text{for all $i\ge 0$.}
  \]
  Now, injectivity in the top degree
  $\times \ell \colon [A']_{\lfloor \frac{e-1}2\rfloor}\longrightarrow
  [A']_{\lfloor \frac{e+1}2\rfloor}$
  implies injectivity in the lower degrees of $A'$. Since $\ell$ is a
  Lefschetz element on $A'$, we have injectivity of the latter map if
  and only if
  \[
  \dim_k [A']_{\lfloor \frac{e-1}2\rfloor} \le \dim_k [A']_{\lfloor
    \frac{e+1}2\rfloor}.
  \]
  Since $\lfloor \frac{e+1}2\rfloor$ is above the middle degree if
  $S\ne \{1,2,\dots,n\}$, we must have a flat top in the Hilbert
  function of $A'$ between degree $e'-\lfloor \frac{e+1}2\rfloor$ and
  degree $\lfloor \frac{e+1}2\rfloor$ in this situation, where $e'$ is
  the socle degree of $A'$. If this forced flat top has length two, we
  must have
  $\lfloor \frac{e+1}2\rfloor-1 =e'-\lfloor \frac{e+1}2\rfloor$ which
  is only possible if $e$ is even and $e'=e-1$. In this case $A''$ is
  generated by one variable $x_i$ with $d_i=2$ and $\ell$ is a
  Lefschetz element on $A$.

  If there is a flat of length at least three, it follows from
  \cite[Theorem 1]{RRR} that one of the generators of the defining
  ideal of $A'$ must have a degree which is above the end of the
  flat. There can be at most one $d_i$ which is greater than
  $\lfloor \frac{e+1}2\rfloor$, so in this case we must have
  $d_n>\lfloor \frac{e+1}2\rfloor$. In this case, $\ell$ is a
  Lefschetz element of $A$.

  We now relate what we have shown with the statements of our
  proposition.

  In the case (\ref{MonCI1}), we get that $\ell$ is a Lefschetz
  element if and only if $d_n$ is the degree of one of the generators
  of the defining ideal of $A'$, which is equivalent to $a_n\neq 0$.

  If $d_n\le \lfloor \frac{e+1}2\rfloor$, the only case when $\ell$ is
  a Lefschetz element and $A'\ne A$ is when $e$ is even and
  $A'' = k[x_i]/\langle x_i^2\rangle$. This shows (\ref{MonCI2}) and
  (\ref{MonCI3}).
\end{proof}

\begin{remark}
  Case~\ref{MonCI2} of Proposition~\ref{monomial ci} shows that the
  non-Lefschetz locus does not need to be unmixed. The smallest
  example is for $d_1=d_2=2$ and $d_3=d_4=3$ where we get
  \[
  \begin{array}{rl}
    I(\mathcal L_I)   & =\langle
                        {a}_{2}^{\hphantom{1}}{a}_{3}^{2}{a}_{4}^{5},
                        {a}_{1}^{\hphantom{1}}{a}_{3}^{2}{a}_{4}^{5},
                        {a}_{1}^{\hphantom{1}}{a}_{2}^{\hphantom{1}}{a}_{3}^{\hphantom{1}}{a}_{4}^{5}, 
                        {a}_{2}^{\hphantom{1}}{a}_{3}^{3}{a}_{4}^{4},
                        {a}_{1}^{\hphantom{1}}{a}_{3}^{3}{a}_{4}^{4}, 
                        {a}_{2}^{2}{a}_{3}^{2}{a}_{4}^{4},
                        {a}_{1}^{\hphantom{1}}{a}_{2}{a}_{3}^{2}{a}_{4}^{4},
                        {a}_{1}^{2}{a}_{3}^{2}{a}_{4}^{4}, 
                        {a}_{1}^{\hphantom{1}}{a}_{2}^{2}{a}_{3}^{\hphantom{1}}{a}_{4}^{4}, 
    \\&
        {a}_{1}^{2}{a}_{2}^{\hphantom{1}}{a}_{3}{a}_{4}^{4},
        {a}_{2}^{\hphantom{1}}{a}_{3}^{4}{a}_{4}^{3}, 
        {a}_{1}^{\hphantom{1}}{a}_{3}^{4}{a}_{4}^{3},
        {a}_{2}^{2}{a}_{3}^{3}{a}_{4}^{3}, 
        {a}_{1}^{2}{a}_{3}^{3}{a}_{4}^{3}, 
        {a}_{1}^{\hphantom{1}}{a}_{2}^{2}{a}_{3}^{2}{a}_{4}^{3},
        {a}_{1}^{2}{a}_{2}{a}_{3}^{2}{a}_{4}^{3},
        {a}_{2}^{\hphantom{1}}{a}_{3}^{5}{a}_{4}^{2}, 
        {a}_{1}^{\hphantom{1}}{a}_{3}^{5}{a}_{4}^{2},	\\&
                                                            {a}_{2}^{2}{a}_{3}^{4}{a}_{4}^{2},
                                                            {a}_{1}^{\hphantom{1}}{a}_{2}^{\hphantom{1}}{a}_{3}^{4}{a}_{4}^{2}, 
                                                            {a}_{1}^{2}{a}_{3}^{4}{a}_{4}^{2}, 
                                                            {a}_{1}^{\hphantom{1}}{a}_{2}^{2}{a}_{3}^{3}{a}_{4}^{2},
                                                            {a}_{1}^{2}{a}_{2}{a}_{3}^{3}{a}_{4}^{2},
                                                            {a}_{1}^{\hphantom{1}}{a}_{2}^{\hphantom{1}}{a}_{3}^{5}{a}_{4}, 
                                                            {a}_{1}^{\hphantom{1}}{a}_{2}^{2}{a}_{3}^{4}{a}_{4},
                                                            {a}_{1}^{2}{a}_{2}{a}_{3}^{4}{a}_{4}^{\hphantom{1}} 
                                                            \rangle
  \end{array}
  \]
  with radical
  $\sqrt{I(\mathcal L_I) )} =\langle a_1a_3a_4,a_2a_3a_4\rangle =
  \langle a_1,a_2\rangle\cap\langle a_3\rangle\cap \langle
  a_4\rangle$.
\end{remark}

\begin{example}
  Proposition~\ref{monomial ci} only gives us that $\mathcal L_I$ is
  defined {\em set-theoretically} by the equation $a_1 \cdots a_n = 0$
  in the cases given by (\ref{MonCI3}). Scheme-theoretically,
  $\mathcal L_I$ is defined by an ideal generated by maximal minors of
  certain matrices as seen in Section~\ref{Sec:Prel}.  For instance,
  if $n=3$ and $d_1=d_2=d_3=4$, the Hilbert function is
  $(1,3,6,10,12,12,10,6,3,1)$ and the defining polynomial of
  $\mathcal L_I$ is $a_1^4 a_2^4 a_3^4$.  More generally, if $n=3$ and
  $d_1=d_2=d_3=d$ where $d$ is even, then the Hilbert function of
  $R/I$ is $(1,h_1, \dots ,h_e)$ with $e=3d-3$ and
  $h_{\frac{3d-4}{2}} = h_{\frac{3d-2}{2}} = 3\left (\frac{d}{2}
  \right)^2$
  and the defining polynomial of $\mathcal L_I$ is
  $(a_1a_2a_3)^{{\left ( \frac{d}{2} \right )}^2}$.
\end{example}

This example leads to the following two immediate corollaries.

\begin{corollary} \label{divisible} If $\langle F_1,\dots,F_n \rangle$
  is any complete intersection in $k[x_1,\dots,x_n]$ with
  $\deg F_1 = \dots = \deg F_n = d$, and if $n(d-1)$ is odd (i.e. if
  $n$ is odd and $d$ is even), then the value of the Hilbert function
  in degrees $\frac{n(d-1)-1}{2}$ and $\frac{n(d-1)+1}{2}$ is
  divisible by $n$.
\end{corollary}

\vskip 2mm
\begin{corollary} \label{NLL for monomials} Let
  $I = \langle x_1^{d_1}, \dots, x_n^{d_n} \rangle$. If
  $d_1 = \dots = d_n = 2$ and $n$ is even then the non-Lefschetz locus
  $\mathcal L_I$ has codimension 2.  In all other cases, it has
  codimension 1.  Furthermore, if $d_1 = \dots = d_n = d$ where $n$ is
  odd and $d$ is even, then
  $I(\mathcal L_I)=(a_1^{\alpha} · \cdots · a_n^{\alpha})$ where
  $\alpha =\frac{1}{n}h_{\frac{n(d-1)-1}{2}}$ and
  $h_{\frac{n(d-1)-1}{2}} = h_{\frac{n(d-1)+1}{2}}$.  When $d=2$, this
  is equal to $\binom{n}{\frac{n-1}{2}}$.
\end{corollary}

\begin{proof}
  The ideas are contained in the proof of Proposition \ref{monomial
    ci}.  In particular, under the hypothesis $d_1 = \dots = d_n = d$
  where $n$ is odd and $d$ is even, the expected codimension of the
  non-Lefschetz locus is achieved, namely codimension 1.  In this case
  the degree of the non-Lefschetz locus is equal to
  $h_{\frac{n(d-1)-1}{2}}$, and the generating polynomial has to be
  symmetric with respect to all $n$ variables. The fact that $\alpha$
  is an integer is guaranteed by Corollary~\ref{divisible}.
\end{proof}

\begin{remark}
  One can also study the non-Lefschetz locus with respect to the
  Strong Lefschetz Property.  Junzo Watanabe has communicated to us
  that he has extended Corollary~\ref{NLL for monomials} for the
  question of the Strong Lefschetz Property, showing that
  $a_1 x_1 + \dots + a_n x_n$ is a Strong Lefschetz element for
  $R/(x_1^2, \dots, x_n^2)$ if and only if $a_1 a_2 \dots a_n \neq 0$.
  Thus the non-Lefschetz locus for $R/(x_1^2, \dots, x_n^2)$ for the
  Strong Lefschetz Property has codimension 1, not 2 as it was for the
  non-Lefschetz locus for the Weak Lefschetz Property.

  Notice that if a linear form $\ell$ is a non-Weak-Lefschetz element
  for $R/I$ then of course it is a non-Strong-Lefschetz element, so
  Watanabe's case is the only one left open by Corollary \ref{NLL for
    monomials}.
\end{remark}

\subsection{Jordan types}

Multiplication by a linear form $\ell$ corresponds to a nilpotent
linear operator on the artinian algebra $A$. The Jordan type of this
nilpotent operator is an integer partition $P_L$ of $\dim_k A$.

The study of Jordan types refines the study of Lefschetz properties as
we have the following:
\begin{itemize}
\item $\ell$ is a weak Lefschetz element if and only if the number of
  parts of $P_L$ equals the maximal value of the Hilbert function of
  $A$.
\item $\ell$ is a strong Lefschetz element if and only if $P_L$ equals
  the dual partition the partition given by the Hilbert function of
  $A$.
\end{itemize}

Here we investigate the possible Jordan types of linear forms for the
case when $A$ is a monomial complete intersection.

For a degree sequence $d_1,d_2,\dots,d_n$ let $P_{d_1,d_2,\dots,d_n}$
denote the dual partition to the partition given by the Hilbert
function of an artinian complete intersection of type
$(d_1,d_2,\dots,d_n)$. For a partition $P$ we denote by $P^k$ the
partition given by repeating all parts of $P$ $k$ times.

\begin{proposition}\label{prop:Jordan} Let
  $A=k[x_1,x_2,\dots,x_n]/\langle
  x_1^{d_1},x_2^{d_2},\dots,x_n^{d_n}\rangle$
  be a monomial complete intersection in characteristic zero. The
  possible Jordan types for linear forms $\ell$ are
  $P_{d_{i_1},d_{i_2},\dots,d_{i_k}}^{m}$, where
  $m = \prod_{j=1}^n d_j/\prod_{j=1}^k d_{i_j}$, for all non-empty subsequences
  ${d_{i_1},d_{i_2},\dots,d_{i_k}}$ of $d_1,d_2,\dots,d_n$.
\end{proposition}

\begin{proof}
  From the action of the torus $(k^\ast)^n$ we see that the Jordan
  type of a linear form $\ell = a_1x_1+a_2x_2+\cdots+a_nx_n$ depends
  only on which coefficients are non-zero. Let
  $\{i_1,i_2,\dots,i_k\}$ be the indices for which the coefficients
  are non-zero and let $\{j_1,j_2,\dots,j_{n-k}\}$ be the remaining
  indices.

  Let $A'$ be the artinian monomial complete intersection of type
  $(d_{i_1},d_{i_2},\dots,d_{i_k})$ and let $A''$ be the artinian
  mononomial complete intersection of type
  $(d_{j_1},d_{j_2},\dots,d_{i_{n-k}})$. We now have
  $A\cong A'\otimes A''$ and $\ell = \sum_{j=1}^k a_{i_j}x_{i_j}$ is a
  strong Lefschetz element acting on the first factor while it acts
  trivially on the second factor. Thus the Jordan type of $\ell$ is
  $P_{d_{i_1},d_{i_2},\dots,d_{i_k}}^m$, where
  $m = \dim_k A'' = = \prod_{j=1}^n d_j/\prod_{j=1}^k d_{i_j}$.
\end{proof}

\begin{example} The situation is easiest to summarize when all
  degrees are equal. Consider for example the case $n=4$ and
  $d_1=d_2=d_3=d_4=2$. There are combinatorially just four possible
  subseqences and the four possible Jordan types are
  \[ [5\,3^3\,1^2],\quad [4\,2^2]^2 = [4^2\,2^4],\quad [3\,1]^4 =
  [3^4\,1^4] \quad\text{and}\quad [2]^8,
  \]
  corresponding to the linear forms $x_1+x_2+x_3+x_4$, $x_1+x_2+x_3$,
  $x_1+x_2$ and $x_1$, respectively.
\end{example}
\section{The non-Lefschetz locus of a general complete
  intersection}\label{sec:general ci}

In the previous section we considered the non-Lefschetz locus of a
monomial complete intersection, and saw that it has codimension 1. We
also know that a general complete intersection has a non-Lefschetz
locus of positive codimension (since the complete intersection has the
WLP, thanks to the main result of \cite{stanley}, \cite{watanabe} and
\cite{RRR}).  The purpose of this section is to describe the precise
codimension of this locus for a general complete intersection.

\begin{notation} \label{nota} We begin in the setting of
  $R = k[x_1,\dots,x_n]$, and then turn to the case
  $n=3,4$. Throughout this section we will fix integers
  $2 \leq d_1 \leq \dots \leq d_n$, and $I$ will be a complete
  intersection ideal, $I = \langle F_1,\dots,F_n \rangle$, where
  $\deg F_i = d_i$ and $F_i$ is a general form of degree $d_i$. We
  will denote by $e$ the socle degree of $R/I$, namely
  $e = (\sum _{i=1}^n d_i) -n$.  We will denote by
  $(1,h_1,\dots,h_{e-1},h_e)$ the $h$-vector (i.e. Hilbert function)
  of $R/I$.
\end{notation}

We will describe the expected codimension of the non-Lefschetz locus
in Conjecture \ref{conj}.  One of our goals is to prove that for
$n =3$ or $4$, and for a general choice of $F_{i}$, $1\le i \le n$,
the non-Lefschetz locus $\mathcal L_I$ of
$I = \langle F_1,\dots,F_n \rangle$ has the expected codimension.

\begin{remark} \label{e even} When the $F_i$ are general, we know that
  $R/I$ has the WLP, so $\mathcal L_I \neq (\mathbb P^{n-1})^\ast$.
  In the case where the socle degree $e$ is odd, the Hilbert function
  of $R/I$ has at least two values in the middle that are equal.
  Thanks to Corollary~\ref{Scheme-theoretic}, this means that
  $\mathcal L_I$ is defined by the vanishing of the determinant of a
  square matrix of size $h_{\frac{e-1}{2}} \times h_{\frac{e+1}{2}}$,
  hence (since $\mathcal L_I \neq (\mathbb P^{n-1})^\ast$)
  $\mathcal L_I$ is a hypersurface of degree
  $\delta_I = h_{\frac{e-1}{2}}$.  So the case of odd socle degree is
  completely understood, and from now on we will assume without loss
  of generality that $e$ is even.
\end{remark}

Based on computer experiments~\cite{M2} and our results in four or
fewer variables, we make the following conjecture.

\begin{conjecture}\label{conj} Let      $I = \langle F_1, \cdots , F_n \rangle \subset R$ be
  a complete intersection ideal of general forms as in Notation
  \ref{nota}, and assume that $e$ is even (see Remark \ref{e
    even}). Then
  \[
  \codim \mathcal L_I = \min \{ h_{\frac{e}{2}}-h_{\frac{e}{2} -1}+1 ,
  n \}
  \]
  where we consider the empty set to have codimension $n$ in
  $\mathbb P^{n-1}$.  In particular,
  $\mathcal L_I\subset (\PP^{n-1})^\ast$ is non-empty if and only if
  $h_{\frac{e}{2}}-h_{\frac{e}{2} -1}\le n-2$ and in that case
  $\delta _I:=\deg(\mathcal L_I)= \binom{ h_{\frac{e}{2}}}
  {h_{\frac{e}{2}}-h_{\frac{e}{2} -1}+1}$.
\end{conjecture}

\begin{remark} Notice that in Conjecture \ref{conj} the hypothesis
  that the complete intersection artinian ideal $I\subset R$ is
  generated by general forms cannot be dropped.  In fact, a complete
  intersection $I\subset k[x_1,x_2,x_3]$ of type $(3,3,3) $ has
  $h$-vector $(1 , 3 , 6 , 7 , 6 , 3 , 1)$, so the expected
  codimension of the non-Lefschetz locus $\mathcal L_I$ is 2; and we
  will see later that indeed it is true for a general choice of 3
  cubics $F_1,F_2,F_3\in k[x_1,x_2,x_3]$ (cf. Theorem
  \ref{LI,n=3,d1d2d3}).  But unfortunately it is not true for {\em
    every } choice.  For instance, we saw in the last section that if
  we take
  $I = \langle F_1,F_2,F_3 \rangle=\langle x_1^3,x_2^3,x_3^3\rangle $
  we get that $\codim \mathcal L_I=1$ since a line
  $a_1x_1+a_2x_2+a_3x_3\in k[x_1,x_2,x_3]$ fails to be a Lefschetz
  element of $k[x_1,x_2,x_3]/\langle x_1^3,x_2^3,x_3^3\rangle $ if and
  only if $a_1a_2a_3=0$.  Therefore, if we fix coordinates $a_1$,
  $a_2$ and $a_3$ in $(\PP^2)^\ast $, the support of $\mathcal L_I$ is
  the union of the lines $\ell _1: a_1=0$, $\ell _2: a_2=0$ and
  $\ell _3: a_3=0$.
\end{remark}

\begin{remark} \label{by one} We will see shortly that to measure the
  non-Lefschetz locus in $(\mathbb P^{(n-1)})^\ast$, it will be enough
  to measure how many such algebras fail the WLP in a suitable
  irreducible parameter space.  As noted in Section~\ref{Sec:Prel}, if
  $R/I$ is a complete intersection and the WLP fails, it must fail "in
  the middle", and possibly also in other degrees.  By semicontinuity
  and under the hypothesis that e is even, to measure the dimension of
  the set of algebras failing the WLP (in an irreducible parameter
  space) we can assume that WLP fails from degree
  $h_{ \frac{e}{2} -1}$ to $h_{ \frac{e}{2} }$ (and, by duality, from
  $h_{ \frac{e}{2} }$ to $h_{ \frac{e}{2} +1 }$), and that the failure
  is just by one.
\end{remark}

\begin{remark}\label{large dn}
  We have $d_1 \leq \dots \leq d_n$.  For large values of $d_n$ the
  question of the non-Lefschetz locus for a general complete
  intersection with generator degrees $d_1,\dots d_n$ is clear.
  \begin{enumerate}
  \item If
    $d_n \geq d_1 + \dots + d_{n-1} - (n-1) + 2 = d_1 + \dots +
    d_{n-1} -n + 3$
    then $h_{\frac{e}{2}-1} = h_{\frac{e}{2}}$ (remembering that we
    are assuming $e$ even), and the conjecture is clear (with the
    non-Lefschetz locus consisting of the linear forms through
    individual points).  \medskip
  \item If
    $d_n = d_1 + \dots + d_{n-1} - (n-1) +1 = d_1 + \dots + d_{n-1} -n
    + 2$
    then $R/(F_1\dots F_{n-1})$ is the coordinate ring of the reduced
    complete intersection set of points, $Z$, in $\mathbb P^{n-1}$
    defined by $(F_1,\dots, F_{n-1})$, which reaches the multiplicity
    in degree $d_1 + \dots + d_{n-1} - (n-1) $.  If $\{h_i\}$ is the
    Hilbert function of $R/(F_1,\dots,F_n)$, then clearly \medskip
    \begin{itemize}
    \item[$\bullet$] $d_1 + \dots + d_{n-1} - (n-1) = \frac{e}{2}$;
      \medskip
    \item[$\bullet$] $h_{\frac{e}{2}} - h_{\frac{e}{2}-1} = 1$;
      \medskip
    \item[$\bullet$] $h_{\frac{e}{2}} = d_1 d_2 \dots d_{n-1}$.
      \medskip
    \item[$\bullet$] The Hilbert function of $R/I$ agrees with that of
      $R/I_Z$ in degrees $\leq \frac{e}{2}$.
    \end{itemize}
    \medskip
    \noindent Notice that $Z$ has the Uniform Position Property, since
    the $F_i$ are general.  We claim that \medskip
    \begin{quotation}
      {\em a linear form $\ell$ fails to have maximal rank from degree
        $\frac{e}{2}-1$ to degree $\frac{e}{2}$ if and only if $\ell$
        vanishes on (any) two points of $Z$.}
    \end{quotation}
    \medskip
    \noindent Indeed, if $P_1, P_2 \in Z$, notice first that the
    Hilbert function of $Z \backslash \{P_1 \}$ agrees with that of
    $Z$ up to and including degree $\frac{e}{2}-1$, and is one less
    than that of $Z$ from then on.  The Hilbert function of
    $Z \backslash \{ P_1,P_2 \}$ agrees with that of $Z$ up to and
    including degree $\frac{e}{2} -2$, is one less than that of $Z$ in
    degree $\frac{e}{2}-1$, and is two less than that of $Z$ from
    degree $\frac{e}{2}$ on. In particular, there is a form of degree
    $\frac{e}{2} -1$ vanishing on all of $Z$ except $P_1 \cup P_2$,
    but the same is not true for all of $Z$ except only $P_1$.

    Since $R/I_Z$ has depth 1, a linear form $\ell$ not vanishing on
    any point of $Z$ is a non-zerodivisor, so the resulting
    multiplication from degree $\frac{e}{2}-1$ to degree $\frac{e}{2}$
    is injective.  If $\ell$ vanishes at just one point, $P_1$, of
    $Z$, then for a form $F$ of degree $\frac{e}{2}-1$,
    $\ell \cdot F = 0$ in $R/I$ means that $F$ vanishes at all points
    of $Z$ except $P_1$.  But we know that any form of degree
    $\frac{e}{2}-1$ vanishing at all but one point must in fact vanish
    on all of $Z$, so $F = 0$ in $R/I$.  On the other hand, any linear
    form vanishing on the line spanned by $P_1$ and $P_2$ lies in the
    non-Lefschetz locus, which then has codimension 2 and degree
    $\binom{h_{\frac{e}{2}}}{2}$ as claimed in Conjecture \ref{conj}.
    \medskip
  \item If $d_n = d_1 + \dots + d_{n-1} -n+1$ then $R/I$ has odd socle
    degree, so the non-Lefschetz locus has codimension 1 and degree
    $d_1 \dots d_{n-1} -1$.  \medskip
  \item Finally, assume that $d_n = d_1 + \dots + d_{n-1}-n$.  In this
    case $\langle F_1,\dots,F_{n-1} \rangle$ defines a complete
    intersection set of $d_1 \cdots d_{n-1}$ points, $Z$, and its
    Hilbert function reaches its multiplicity in degree
    $d_1 + \dots + d_{n-1} - n+1 = \deg F_n +1$.  More precisely,
    letting $s = d_1 + \dots + d_{n-1}$ and $d = d_1 \cdots d_{n-1}$,
    its Hilbert function is \medskip
    \[
    \begin{array}{c|ccccccccccccccc}
      \text{degree} & 0 & 1 & 2 & \dots & (s-n-1) & (s-n) & (s-n+1) & (s-n+2) & \dots  \\ \hline
                    & 1 & n & h_2 & \dots & d-n & d-1 & d & d &\dots 
    \end{array}
    \]
    \noindent and the Hilbert function of $R/I$ is
    \[
    \begin{array}{c|ccccccccccccccc}
      \text{degree} & 0 & 1 & 2 & \dots & (s-n-1) & (s-n) & (s-n+1) & \dots & e-1 & e  \\ \hline
                    & 1 & n & h_2 & \dots & d-n & d-2 & d-n & \dots & n & 1   
    \end{array}
    \]
    \noindent For a linear form $\ell \in R$, the failure of
    $\times \ell : [R/I]_{s-n-1} \rightarrow [R/I]_{s-n}$ to be
    injective is equivalent to the condition that the restriction
    $\bar F_n$ of $F_n$ to $R/\langle \ell \rangle$ is in the
    restricted ideal
    $\langle \bar F_1 , \dots , \bar F_{n-1} \rangle$. Since $I$ is
    artinian, it follows then that
    $\langle \bar F_1,\dots, \bar F_{n-1} \rangle$ is a complete
    intersection. In particular, $\ell$ is a non-zerodivisor on
    $R/\langle F_1,\dots,F_{n-1} \rangle$. We also note that in this
    situation, the conjectured codimension of $\mathcal L_I$ is
    $(d-2) - (d-n) +1 = n-1$ in $(\mathbb P^{n-1})^\ast$, i.e. there
    should only be a finite number of linear forms failing to induce
    an injective homomorphism from degree $s-n-1$ to degree $s-n$.
  \end{enumerate}

  Thus from now on we may assume that
  $d_n \leq d_1 + \dots + d_{n-1} - n$, and if equality holds we have
  an equivalent condition for failure to have maximal rank.
\end{remark}

Our goal in this section is to prove Conjecture \ref{conj} in the
cases $n=3$ and $n=4$.  We begin with a description of the approach
that we will take except for Theorem \ref{LI,n=3,d1d2d3}.  Fix degrees
$d_1,\dots,d_n$ for the complete intersections in
$R = k[x_1,\dots,x_n]$, with $d_1 \leq d_2 \leq \dots \leq d_n$.  Let
$CI(d_1,\dots,d_n)$ be the irreducible space parametrizing all such
complete intersections.  Let $(\mathbb P^{n-1})^\ast$ be the
projective space parametrizing the linear forms of $R$ (up to scalar
multiple).  For each complete intersection $I$ and linear form
$\ell $, we consider the pair
$(\ell ,I) \in (\mathbb P^{n-1})^\ast \times CI(d_1,\dots,d_n)$.  Let
$X$ be the set of such pairs such that $\ell $ is not a Lefschetz
element for $A= R/I$.

Since the $d_i$ are given, there is a precise degree where this latter
condition must be checked: $(\ell ,I) \in X$ if and only if
$\times \ell : [R/I]_{\frac{e}{2}-1} \rightarrow [R/I]_{\frac{e}{2}}$
fails to be injective (recall that the socle degree $e$ is assumed to
be even, thanks to Remark \ref{e even}).  Since the general element of
$CI(d_1,\dots,d_n)$ has the WLP, there are expected values for the
Hilbert function of $R/(I,\ell )$ in degrees $\frac{e}{2}$ and
$\frac{e}{2}+1$ (the latter being 0), and $(\ell , I ) \in X$ if and
only if these values are not achieved.

Consider the projections $\phi_1$ and $\phi_2$:
\begin{equation} \label{big picture}
  \begin{tikzpicture} [baseline=(current  bounding  box.center)]
    \node (v1) at (1,0) {$(\ell ,I)$}; \node (v2) at (2,0)
    {$\in$}; \node (v3) at (4.8,0)
    {$(\mathbb P^{n-1})^\ast \times
      CI(d_1,\dots,d_n)$}; \node (v4) at (7.6,0)
    {$\supset$}; \node (v5) at (8.3,0) {$X$};
    \node (v6) at (3,-1.3) {$(\mathbb
      P^{n-1})^\ast$}; \node (v7) at (6,-1.3)
    {$CI(d_1,\dots,d_n)$}; \draw [->] (v3) -- (v6); \draw [->] (v3) --
    (v7); \node at (3.5,-.55)
    {$\phi_1$}; \node at (5.8,-.55) {$\phi_2$};
  \end{tikzpicture}
\end{equation}

\noindent We need to show that there is a non-empty open set
$U \subset CI(d_1,\dots,d_n)$ such that if $I \in U$ then the closure
of $\phi_1(\phi_2^{-1}(I) \cap X)$ has the expected codimension as
described in Conjecture~\ref{conj}.  Thus we want to show that the
intersection of $X$ with the generic fibre of $\phi_2$ has the
expected dimension (computed from Conjecture \ref{conj}).  More
precisely, let
$m = (n-1) - \min \{ h_{\frac{e}{2}} - h_{\frac{e}{2}-1} +1,n \}$, the
expected dimension of $\mathcal L_I$, and let $I$ be a general element
of $CI(d_1,\dots,d_n)$.  Then Conjecture \ref{conj} says that
\begin{equation} \label{1st to show} \dim (\phi_2^{-1} (I)) \cap X =
  m.
\end{equation}
We will reformulate this.  Let $p = \dim CI(d_1,\dots,d_n)$.  We want
to show that there is an open subset $U \subset CI(d_1, \dots ,d_n)$
such that
\[
\dim ( \phi_2 ^{-1}(U) \cap X ) = m+p.
\]
Now, $\phi_1$ is surjective, and the fibres all have the same
dimension (since we can always do a change of variables).  Thus we
want to show that for any linear form $\ell $ (viewed as an element of
$(\mathbb P^{n-1})^\ast$),
\[
\dim ( \phi_2^{-1} (U) \cap X \cap \phi_1^{-1}(\ell )) = m+p-(n-1).
\]

So from now on we fix a linear form $\ell $.  We denote by
$ACI_{\ell }(d_1, \dots,d_n)$ the irreducible space of ideals in
$S = R/(\ell )$ with generators in degrees $d_1, \dots d_n$.  We note
that an ideal in $ACI_{\ell }(d_1,\dots,d_n)$ may have only $n-1$
minimal generators.  This would happen for instance if
$d_n > d_1 + \dots + d_{n-1} - (n-1)$ and $\ell$ is a non-zerodivisor
on $R/\langle F_1,\dots, F_{n-1} \rangle$, but we have assumed this
not to be the case in Remark \ref{large dn}.  But even avoiding this
situation, it may happen that an ideal in $CI(d_1,\dots,d_n)$
restricts to an ideal in $ACI_{\ell }(d_1, \dots,d_n)$ with only $n-1$
minimal generators.  Let $V \subset ACI_{\ell }(d_1,\dots,d_n)$ be the
open subset consisting of restricted ideals
$(\bar F_1,\dots, \bar F_n)$ such that all the $\bar F_i$ are minimal
generators.

Consider the morphism
\begin{equation} \label{big picture 2}
  \begin{tikzpicture} [baseline=(current  bounding  box.center)]
    \node (v3) at (4.8,0)
    {$CI(d_1,\dots,d_n)$}; \node (v6) at (4.8,-1.3)
    {$ACI_{\ell
      }(d_1,\dots,d_n)$.}; \draw [->] (v3) -- (v6); \node at (5.1,-.6)
    {$\phi$};
  \end{tikzpicture}
\end{equation}

We want to study a certain subvariety,
$Y \subset ACI_\ell(d_1,\dots,d_n)$. The precise definition of $Y$
will depend on the value of $d_n$, breaking into two cases, but the
treatment of $Y$ will be the same in both cases.

\medskip

\noindent \underline{Case 1}: $d_n = d_1 + \dots + d_{n-1} - n$.
\medskip We have seen in Remark \ref{large dn} (4) that in this case
$m=0$, and that failure of maximal rank is equivalent to
$\bar F_n \in \langle \bar F_1,\dots, \bar F_{n-1} \rangle$, which
then is a complete intersection.  By Remark \ref{by one}, or by direct
observation in this case, we can assume that the Hilbert function of
the restricted ideal differs by one, in degrees
$d_1 + \dots + d_{n-1} - n$ and $d_1 + \dots + d_{n-1} - (n-1)$, from
the expected one.  Let $Y \subset ACI_\ell (d_1,\dots,d_n)$ be the
subset in the complement of $V$ consisting of those ideals such that
the first $n-1$ generators form a regular sequence, and the last
generator is not minimal.

\medskip

\noindent \underline{Case 2}: $d_n < d_1 + \dots + d_{n-1} - n$.
\medskip In this case we let $Y \subset V$ be the set of ideals
$\bar I$ such that $h_{S/\bar I}(\frac{e}{2}+1) > 0$.  (The
distinction between the cases is that the ideals of $Y$ are complete
intersections in Case 1, and are not complete intersections in Case
2.)

\medskip

Notice that in both cases,
\[
\dim \phi^{-1} (Y) = \dim ( \phi_2^{-1} (U) \cap X \cap
\phi_1^{-1}({\ell })).
\]
Notice also that the fibres of $\phi$ over $V \cup Y$ all have the
same dimension, namely
\[
p - \dim ACI_{\ell } (d_1, \dots,d_n).
\]
So we want to show that
\[
\dim Y + p - \dim ACI_{\ell } (d_1, \dots,d_n) = m + p - (n-1),
\]
i.e. that
\[
\dim Y = m -(n-1) + \dim ACI_{\ell } (d_1, \dots,d_n).
\]
Equivalently, we want to show that
\begin{equation} \label{to show} \hbox{\em The codimension of $Y$ in
    $ACI_{\ell }(d_1,\dots,d_n)$ is
    $\min \{ h_{\frac{e}{2}} - h_{\frac{e}{2}-1}+1,n \}$.}
\end{equation}
This is what we will prove in the results below.

We have noted above that without loss of generality we can assume that
$d_n \leq d_1 + \dots + d_{n-1} - n$, and that the case of equality is
handled slightly differently from the case of strict inequality.  We
now consider equality.

\begin{proposition} \label{special case of conj} Let
  $I = \langle F_1,\dots,F_n \rangle \subset R = k[x_1,\dots,x_n]$ be
  a complete intersection generated by general forms of degrees
  $2 \leq d_1 \leq d_2 \leq \dots \leq d_n$.  Assume that
  $d_n = d_1 + \dots + d_{n-1} -n$.  Then Conjecture \ref{conj} is
  true.
\end{proposition}

\begin{proof}
  We have defined the quasi-projective variety $Y$ in Case 1 above.
  From what we said in Remark \ref{large dn} (4) and in Case 1 of the
  discussion above, we want to show that the codimension of $Y$ in
  $ACI_\ell(d_1,\dots,d_n)$ is $n-1$. We recall that a complete
  intersection of type $(d_1,\dots,d_{n-1})$ in
  $R/\langle \ell \rangle $ with $d_1 \geq 2$ has Hilbert function
  with value $n-1$ in degree $d_n = d_1 + \dots + d_{n-1} -n$.

  Now, let $\mathcal M_\ell (d_1,\dots,d_{n-1})$ be the variety
  parametrizing the ideals with generator degrees $d_1,\dots,d_{n-1}$,
  and let $U' \subset \mathcal M_\ell$ be the dense open subset
  consisting of complete intersections of type $(d_1,\dots,d_{n-1})$.
  Consider
  \begin{equation} \label{case 1 diag}
    \begin{tikzpicture} [baseline=(current  bounding  box.center)]
      \node (v1) at (1.5,0) {$Y$}; \node (v2) at (2,0)
      {$\subseteq$}; \node (v3) at (4,0)
      {$ACI_\ell (d_1,\dots,d_{n})$}; \node (v6) at (1.5,-1.3)
      {$U'$}; \node (v8) at (2,-1.3)
      {$\subseteq$}; \node (v7) at (4,-1.3)
      {$\mathcal
        M_\ell(d_1,\dots,d_{n-1})$ . }; \draw [->] (v3) -- (v7); \node
      at (4.4,-.6) {$\phi$};
    \end{tikzpicture}
  \end{equation}
  \noindent We have that $Y$ is contained in $\phi^{-1} (U')$, and
  $\phi^{-1} (U')$ is a dense open subset of \linebreak
  $ACI_\ell (d_1,\dots,d_n)$.  For any $J \in U'$, the codimension of
  $\phi^{-1} (J) \cap Y$ in $\phi^{-1} (J)$ is $n-1$, thanks to the
  Hilbert function observation above.  The desired conclusion (\ref{to
    show}) follows from this.
\end{proof}

Thus from now on we can assume that $d_n < d_1 + \dots + d_{n-1}-n$,
and that we are in Case 2 above.  To fix the ideas for most of the
rest of the paper in a simple first case, we first state the case
$n=3$ and $\deg(F_i) = d$ for $1 \leq i \leq 3$ and prove the
analogous case $n=4$ and $\deg (F_i) = d$ for $1 \leq i \leq 4$.

\begin{proposition} \label{LI,n=3} Let
  $I = \langle F_1, F_2 , F_3 \rangle \subset R=k[x_1,x_2,x_3]$ be a
  complete intersection generated by general forms of degree
  $(d,d,d)$, $d\ge 2$. Then we have
  \begin{itemize}
  \item[(1)] If $e$ is odd then $\codim \mathcal L_I =1$ and
    $\mathcal L_I\subset (\PP^{2})^\ast$ is a curve of degree
    $3(\frac{d}{2})^2$.
  \item[(2)] If $e$ is even then
    $\codim \mathcal L_I =h_{\frac{e}{2}}-h_{\frac{e}{2} -1}+1 =2$ and
    $\deg \mathcal L_I = \binom{h_{\frac{e}{2}}}{2} =
    \binom{\frac{3d^2+1}{4}}{2}$.
  \end{itemize}
\end{proposition}

\begin{proposition} \label{LI,n=4} Let
  $I = \langle F_1, F_2 , F_3 , F_4 \rangle \subset
  R=k[x_1,x_2,x_3,x_4]$
  be a complete intersection generated by general forms of degree
  $(d,d,d,d)$, $d\ge 2$.
  \begin{itemize}
  \item[(1)] If $d=2$ then $\codim \mathcal L_I=3$, and in particular
    $\mathcal L_I\subset (\PP^{3})^\ast$ is a set of $20$ different
    points.
  \item[(2)] If $d\ge 3$ then $\mathcal L_I =\emptyset $.
  \end{itemize}
\end{proposition}

\begin{proof} (1) The $h$-vector of $R/I$ is $(1,4,6,4,1)$ and
  $\mathcal L_I$ is a scheme defined by the maximal minors of a
  $4\times 6$ matrix with linear entries. We will prove that
  $\mathcal L_I \subset (\mathbb P^3)^\ast$ has codimension 3 and
  consists of 20 different points which shows that $\mathcal L_I$ is a
  standard determinantal scheme.  If $\ell $ fails to give an
  injection from degree~1 to degree 2, then there is a linear form $M$
  such that $\ell M \in I$.  So we first want to know how many
  reducible quadrics lie in the projectivization of the 4-dimensional
  vector space generated by $F_1,F_2,F_3,F_4$ inside
  $\mathbb P [R]_2 = \mathbb P^9$.  The dimension of the space of such
  reducible quadrics is 6, and its degree is 10 (\cite{harris}, top of
  page 300). Thus its intersection with a general 3-dimensional linear
  space in $\mathbb P^9$ is a set of 10 points in $\mathbb P^9$.  Such
  a linear space avoids the locus of double planes, and each of the 10
  points is of the form $\ell _1 \ell _2$ where either $\ell _1$ or
  $\ell _2$ could play the role of $\ell $ for us.  Thus there are 20
  such linear forms, or 20 points in $(\mathbb P^3)^\ast$.
  
  \medskip
  
  \noindent (2) Let $(1,h_1,h_2, \cdots ,h_{e-1},h_e)$ be the
  $h$-vector of $R/I$. Therefore, $e=4d-4$.

  \medskip
  
  \noindent \underline{Claim:} $h_{\frac{e}{2}}- h_{\frac{e}{2}-1}=d.$
  
  \medskip
We will prove a more general result in Lemma~\ref{HF,n=4,d1d2d3d4}, but here we give a completely different proof to illustrate a different approach.

\medskip
  
  \noindent \underline{Proof of the Claim:} We consider the rank 3
  vector bundle ${\mathcal E}$ on $\PP^3$
  \[
  {\mathcal E}:= \ker({\mathcal O}_{\PP^3}(-d)^4
  \stackrel{(F_1,F_2,F_3,F_4)}{\longrightarrow} {\mathcal
    O}_{\PP^3}).\] Using the exact sequences
  \[
  0 \longrightarrow {\mathcal E} \longrightarrow {\mathcal
    O}_{\PP^3}(-d)^4 \longrightarrow {\mathcal
    O}_{\PP^3}\longrightarrow 0, \text{ and}\]
  \[ 0 \longrightarrow {\mathcal O}_{\PP^3}(-4d) \longrightarrow
  {\mathcal O}_{\PP^3}(-3d)^4\longrightarrow {\mathcal
    O}_{\PP^3}(-2d)^6 \longrightarrow {\mathcal E}\longrightarrow 0,\]
  we get
  \[\begin{array}{lll} H^0(\PP^3, {\mathcal E}(t))=0 & \text{ for
                                                       all } & t<2d
      \\ H^2(\PP^3,   {\mathcal E}(t))=0 & \text{ for all } & t\in
                                                              \mathbb Z\\
      H^3(\PP^3,   {\mathcal E}(t))=0  & \text{ for all } & t\ge d-3. \end{array}\]
    Therefore, we have
    \[
    \begin{array}{lll}
      h_{\frac{e}{2}}- h_{\frac{e}{2}-1} & = &
                                               h^1(\PP^3,   {\mathcal E}(2d-2) )
                                               -h^1(\PP^3, {\mathcal E}(2d-3)) \\
                                         & = & -\chi({\mathcal
                                               E}(2d-2))+
                                               \chi({\mathcal
                                               E}(2d-3)) =d 
    \end{array}
    \]
    
    \noindent where the last equality follows applying the
    Riemann-Roch Theorem, and the Claim is proved.
    
    Since $h_{\frac{e}{2}}- h_{\frac{e}{2}-1}=d$, $\mathcal L_I$ is
    expected to be empty and this is what we will prove.  To this end,
    we set
    $S=k[x_1,x_2,x_3,x_4]/(\ell ) \cong
    k[\overline{x_1},\overline{x_2},\overline{x_3}]$
    where $\ell =a_1x_1+a_2x_2+a_3x_3+a_4x_4\in [R]_1$ is a linear
    form.  Call ${\mathcal A}_{d,d,d,d}$ the set of almost complete
    intersection ideals $J\subset S$ of type $(d,d,d,d)$. It holds
    that
    \[
    \dim {\mathcal A}_{d,d,d,d}=\dim Gr \left (4, \binom{d+2}{2}
    \right) )= 4\binom{d+2}{2}-16=2d^2+6d-12.
    \]
    Denote by ${\mathcal B}_{d,d,d,d}$ the set of almost complete
    intersection ideals $J\subset S$ of type $(d,d,d,d)$ and
    $h$-vector
    $(1, h_1-1,h_2-h_1, \cdots , h_{\frac{e}{2}-1}-h_{\frac{e}{2}-2},
    h_{\frac{e}{2}}-h_{\frac{e}{2}-1}+1=d+1, 1)$.
    A general ideal $J$ in ${\mathcal B}_{d,d,d,d}$ can be linked by
    means of a complete intersection $J'$ of type $(d,d,d)$ to a
    Gorenstein ideal $J_1$ with socle degree $2d-3$ and $h$-vector
    \[(1,3,6,\cdots ,\binom{d-1}{2}, \binom{d}{2}-1, \binom{d}{2}-1,
    \binom{d-1}{2},\cdots ,6,3,1).\] Observe that
    \begin{itemize}
    \item[(i)] the dimension of the Gorenstein ideals $J_1$ with
      $h$-vector
      \[
      \left (1,3,6,\cdots ,\binom{d-1}{2}, \binom{d}{2}-1,
        \binom{d}{2}-1, \binom{d-1}{2},\cdots ,6,3,1 \right )
      \]
      is $\binom{2d-1}{2}-d-2 = 2d^2 -4d-1$
      (see    \cite[Example 5.2]{CV}), \\
    \item[(ii)] the dimension of complete intersections $J'$ of type
      $(d,d,d)$ contained in $J_1$ is
      $\dim Gr(3,3d)=3(3d-3)$  (note that $\dim[J_1]_d = \binom{d+2}{2} - \binom{d-1}{2} = 3d$), and \\
    \item[(iii)] the dimension of complete intersections of type
      $(d,d,d)$ contained in $J$ is $\dim Gr(3,4)=3$.
    \end{itemize}
    \medskip To compute $\dim \mathcal B_{d,d,d,d}$ we use liaison.
    The computation is
    \[
    \dim {\mathcal B}_{d,d,d,d}= \left ( \binom{2d-1}{2}-d-2 \right
    )+(9d-9)-3= 2d^2+5d-13.
    \]
    We have only to justify subtracting the value from (iii) in this
    computation.  Indeed, this is to remove over-counting, since the
    same ideal $J$ can be reached from many different ideals $J_1$
    using different complete intersections in $J$.  Now subtracting,
    we see that the difference of the dimensions is
    \[
    (2d^2 + 6d - 12) - (2d^2 + 5d - 13) = d +1 = h_{\frac{e}{2}} -
    h_{\frac{e}{2}-1} + 1.
    \]
    Since this is $> n-1$ for $d \geq 3$, the locus is empty according
    to (\ref{to show}).
  \end{proof}

  \begin{theorem} \label{LI,n=3,d1d2d3} Let
    $I = \langle F_1, F_2 , F_3 \rangle \subset R = k[x_1,x_2,x_3]$ be
    a complete intersection artinian ideal generated by general forms
    of degree $(d_1,d_2,d_3)$. Assume that $d_1\le d_2 \le d_3$. Let
    $e$ be the socle degree of $R/I$ and let
    $(1,h_1,\cdots , h_{e-1},h_e)$ be the $h$-vector of $R/I$.  Then
    \begin{itemize}
    \item[(1)] If $e$ is odd then $\codim \mathcal L_I =1$ and
      $\mathcal L_I\subset (\PP^{2})^\ast$ is a plane curve of degree
      \[
      \begin{cases} d_1d_2 & \text{ if }  d_3\ge d_1+d_2 \\ \\
        \displaystyle d_1d_2-\frac{(d_1+d_2-d_3)^2}{4} =
        \frac{2d_1d_2+2d_1d_3+2d_2d_3 -d_1^2-d_2^2-d_3^2}{4} & \text{
          if } d_3< d_1+d_2.
      \end{cases}
      \]\item[(2)]
      If $e$ is even then
      \[
      \codim \mathcal L_I =h_{\frac{e}{2}}-h_{\frac{e}{2} -1}+1=
      \begin{cases}
        1 \text{ if } d_3 \ge d_1+d_2+1,  \\
        2 \text{ if } d_3\le d_1+d_2-1.
      \end{cases}
      \]
      Moreover, if $d_3\ge d_1+d_2+1$ then
      $\mathcal L_I\subset (\PP^{2})^\ast$ is a plane curve of degree
      $d_1d_2$; and if $d_3\le d_1+d_2-1$ then
      $\mathcal L_I\subset (\PP^{2})^\ast$ is a finite set of
      $\binom{n_I}{2}$ points, where
      \[
      n_I = \frac{2d_1d_2+2d_1d_3+2d_2d_3 +1 -d_1^2-d_2^2-d_3^2}{4}
      \]
    \end{itemize}
  \end{theorem}

  \begin{proof} It is well known that $I$ has WLP and hence
    $\codim \mathcal L_I\ge 1$.

    (1) If $d_3\ge d_1+d_2$ arguing as in Remark \ref{large dn} we see
    that $J=\langle F_1,F_2\rangle$ is the ideal of a set of $d_1d_2$
    different points in $\PP^2$,
    $h_{\frac{e-1}{2}}=h_{\frac{e+1}{2}}=d_1d_2$ and
    \[
    \times \ell : (R/I)_{\frac{e-1}{2}} \longrightarrow
    (R/I)_{\frac{e+1}{2}}
    \]
    with $\ell =ax+by+cz$ fails to be injective if and only if $\ell $
    passes through one of the $d_1d_2$ points defined by
    $J$. Therefore, $\codim (\mathcal L_I)=1$ and
    $\deg(\mathcal L_I)=d_1d_2$.

    Assume $d_3<d_1+d_2$.  In this case we consider the syzygy bundle
    associated to $I$, i.e. the rank 2 vector bundle ${\mathcal E}$ on
    $\PP^2$ defined by
    \[{\mathcal E} := \ker(\oplus _{i=1}^3 {\mathcal O}_{\PP^2}(-d_i)
    \stackrel{(F_1,F_2,F_3)}{\longrightarrow} {\mathcal O}_{\PP^2}).
    \]
    By \cite[Corollary 2.7]{bs}, ${\mathcal E}$ is $\mu $-stable.  By
    \cite[Theorem 2.2]{BK}, the linear form $\ell =ax+by+cz$ fails to
    be a Lefschetz element of $I$ if and only if $\ell =0$ is a
    jumping line of ${\mathcal E}$ if and only if
    ${\mathcal E}_{|\ell } \cong {\mathcal O}_{\ell }(a^1_{\ell
    })\oplus {\mathcal O}_{\ell }(a^2_{\ell })$
    with $|a_{\ell }^1-a_{\ell }^2|\ge 2$.  Since the first Chern
    class $c_1({\mathcal E}(\frac{d_1+d_2+d_3}{2}))=0$, we can apply
    \cite[Theorem 2.2.3]{OSS}, and we get that the set
    $J_{\mathcal E}$ of jumping lines of ${\mathcal E}$ is a curve of
    degree $c_2( {\mathcal E}(\frac{d_1+d_2+d_3}{2}))$ in
    $(\PP^2)^\ast$. Therefore, the non-Lefschetz locus $\mathcal L_I$
    of $I$ is a plane curve of degree
    \begin{multline*}
      c_2 \left ( {\mathcal E}(\frac{d_1+d_2+d_3}{2}) \right ) =
      \frac{(d_1+d_2-d_3)(d_1-d_2+d_3)}4+\frac{(d_1+d_2-d_3)(-d_1+d_2+d_3)}4\\
      +\frac{ (d_1-d_2+d_3)(-d_1+d_2+d_3)}{4}=
      \frac{2d_1d_2+2d_1d_3+2d_2d_3 -d_1^2-d_2^2-d_3^2}{4}.
    \end{multline*}
    \vskip 2mm

    (2) If $d_3\ge d_1+d_2+1$ the result follows from Remark
    \ref{large dn}.  So, let us assume that $d_3\le d_1+d_2-1$.  Let
    $(1,h_1,h_2, \cdots ,h_{e-1},h_e)$ be the $h$-vector of $R/I$.

  \noindent\underline{Claim:} $h_{\frac{e}{2}}- h_{\frac{e}{2}-1}=1.$
  \medskip To prove the claim, we consider the rank $2$ vector bundle
  ${\mathcal E}$ on $\PP^2$
  \[ {\mathcal E} := \ker \left ( \oplus _{i=1}^3 {\mathcal
      O}_{\PP^2}(-d_i) \stackrel{(F_1,F_2,F_3)}{\longrightarrow}
    {\mathcal O}_{\PP^2}\right ).
  \]
  By \cite[Corollary 2.7]{bs}, ${\mathcal E}$ is $\mu $-stable.  Using
  the fact that $ {\mathcal E}$ is a $\mu $-stable rank 2 vector
  bundle on $\PP^2$, $c_1({\mathcal E})=-d_1-d_2-d_3$ and
  ${\mathcal E}_{norm}={\mathcal E}(\frac{d_1+d_2+d_3-1}{2})$, we get
  \[H^0(\PP^2, {\mathcal E}(h_{\frac{e}{2}})=H^2(\PP^2, {\mathcal
    E}(h_{\frac{e}{2}})) =H^0(\PP^2, {\mathcal E}(h_{\frac{e}{2}-1}))=
  H^2(\PP^2, {\mathcal E}(h_{\frac{e}{2}-1}))=0.\] Therefore, we have
  \[
  \begin{array}{lll}
    h_{\frac{e}{2}}- h_{\frac{e}{2}-1} 
    & = &
          h^1(\PP^2,   {\mathcal E}(h_{\frac{e}{2}}) )
          -h^1(\PP^2, {\mathcal E}(h_{\frac{e}{2}-1})) \\
    & = & -\chi({\mathcal E}(h_{\frac{e}{2}}))+ \chi({\mathcal E}(h_{\frac{e}{2}-1})) =1 \end{array}
  \]
  \noindent where the last equality follows applying the Riemann-Roch
  Theorem, and the claim is proved.
  
  Thanks to the claim, the expected codimension of
  $\mathcal L_I\subset (\PP^2)^\ast$ is two and, in fact, we are going
  to prove that $\mathcal L_I\subset (\PP^2)^\ast$ is a set of
  $\binom{n_I}{2}$,
  $n_I:=\frac{2d_1d_2+2d_1d_3+2d_2d_3+1-d_1^2-d_2^2-d_3^2}{4}$,
  different points. To this end, we consider the rank 2 vector bundle
  ${\mathcal E}$ on $\PP^2$
  \[ {\mathcal E} := \ker \left ( \oplus _{i=1}^3 {\mathcal
      O}_{\PP^2}(-d_i) \stackrel{(F_1,F_2,F_3)}{\longrightarrow}
    {\mathcal O}_{\PP^2}\right ).
  \]
  By \cite[Corollary 2.7]{bs}, ${\mathcal E}$ is $\mu $-stable.  By
  \cite[Theorem 2.2]{BK}, the linear form $\ell =ax+by+cz$ fails to be
  a Lefschetz element of $I$ if and only if $\ell =0$ is a jumping
  line of ${\mathcal E}$ if and only if
  ${\mathcal E}_{|\ell } \cong {\mathcal O}_{\ell }(a^1_{\ell })\oplus
  {\mathcal O}_{\ell }(a^2_{\ell })$
  with $|a_{\ell }^1-a_{\ell }^2|\ge 2$.  Since the first Chern class
  $c_1({\mathcal E}(\frac{d_1+d_2+d_3-1}{2}))=-1$, we can apply
  \cite[Corollary 10.7.1]{hulek}, and we get that ${\mathcal E}$ has
  exactly $\binom{c_2({\mathcal E}(\frac{d_1+d_2+d_3-1}{2})) }{2}$
  jumping lines.  Let us compute
  $c_2({\mathcal E}(\frac{d_1+d_2+d_3-1}{2}))$.  From the exact
  sequence
  \[
  0 \longrightarrow {\mathcal E} \longrightarrow \oplus _{i=1}^3
  {\mathcal O}_{\PP^2}(-d_i) \longrightarrow {\mathcal O}_{\PP^2}
  \longrightarrow 0
  \]
  we get that $c_1({\mathcal E})=-d_1-d_2-d_3$ and
  $c_2({\mathcal E})=d_1d_2+d_1d_3+d_2d_3$.  Since
  \[
  c_2 \left ( {\mathcal E} \left ( \frac{d_1+d_2+d_3-1}{2} \right )
  \right ) = c_2({\mathcal E})+c_1 ({\mathcal E}) \left (
    \frac{d_1+d_2+d_3-1}{2} \right )+ \left ( \frac{d_1+d_2+d_3-1}{2}
  \right )^2,
  \]
  we have
  \[
  c_2 \left ({\mathcal E} \left ( \frac{d_1+d_2+d_3-1}{2} \right )
  \right )= \frac{2d_1d_2+2d_1d_3+2d_2d_3+1-d_1^2-d_2^2-d_3^2}{4}
  \]
  and we conclude that the set $J_{\mathcal E}$ of jumping lines of
  ${\mathcal E}$ is a set of $\binom{n_I}{2}$ points in
  $(\PP^2)^\ast$, where
  \[
  n_I:=\frac{2d_1d_2+2d_1d_3+2d_2d_3+1-d_1^2-d_2^2-d_3^2}{4},
  \]
  which proves what we want.
\end{proof}

Our next goal is to prove Conjecture \ref{conj} for $n=4$. To this end
the following lemmas will be very useful.

\begin{lemma}\label{HF,n=4,d1d2d3d4}
  Let
  $I = \langle F_1, F_2 , F_3 ,F_4 \rangle \subset R =
  k[x_1,x_2,x_3,x_4]$
  be a complete intersection artinian ideal generated by general forms
  of degree $(d_1,d_2,d_3,d_4)$. Assume that
  $d_1\le d_2 \le d_3\le d_4$.  Let $e=d_1+d_2+d_3+d_4-4$ be the socle
  degree of $R/I$ and let $(1,h_1,\cdots , h_{e-1},h_e)$ be the
  $h$-vector of $R/I$.  Then
  \begin{enumerate}
  \item If $e$ is odd then
    \[
    h_{\frac{e-1}{2}}= h_{\frac{e+1}{2}}= d_1d_2d_3 -
    \frac14\binom{d_1+d_2+d_3-d_4+1}{3}+\frac14\binom{-d_1+d_2+d_3-d_4+1}{3}
    \]
  \item If $e$ is even then
    \[
    h_{\frac{e}{2}}-h_{\frac{e}{2}-1}= \begin{cases} 0 & \text{ if }
      d_4\ge d_1+d_2+d_3 \\ \frac{d_1+d_2+d_3-d_4}{2} & \text{ if }
      -d_1+d_2+d_3\le d_4\le d_1+d_2+d_3 \\ d_1 & \text{ if } d_4\le
      -d_1+d_2+d_3 .\end{cases}
    \]
  \end{enumerate}
\end{lemma}

\begin{proof} 
  Let $h'_i$ be the $h$-vector of $R/\langle f_1,f_2,f_3\rangle$. Then
  for $e$ odd we get
  \begin{multline*}
    h_{\frac{e-1}2} = \sum_{i=0}^{d_4-1} h'_{\frac{e-1}2-i} =
    \sum_{j=\frac{e+1}2-d_4}^{\frac{e-1}2} h'_j = d_1d_2d_3 -
    \sum_{j=0}^{\frac{e-1}2-d_4}h'_j -
    \sum_{j=\frac{e+1}2}^{d_1+d_2+d_3-3} h'_j \\=
    d_1d_2d_3-\sum_{j=0}^{\frac{d_1+d_2+d_3-d_4-5}2} h'_j -
    \sum_{j=0}^{d_1+d_2+d_3-3-\frac{e+1}2}h'_j
    \\=d_1d_2d_3-\sum_{j=0}^{\frac{d_1+d_2+d_3-d_4-5}2} h'_j -
    \sum_{j=0}^{\frac{d_1+d_2+d_3-d_4-7}2}h'_j
  \end{multline*}
  In the range $0\le j \le \frac{d_1+d_2+d_3-d_4-5}2$ we have that
  $h'_j = \binom{j+2}2- \binom{j-d_1+2}2$ since
  $d_2> \frac{d_1+d_2+d_3-d_4-5}2$. For any $m$ we have that
  \[
  \sum_{j=0}^m \binom{j+2}2 + \sum_{j=0}^{m-1} \binom{j+2}2 =
  \binom{m+3}3+\binom{m+2}3 = \frac14\binom{2m+4}3.
  \]
  Thus we conclude that the maximum value of the $h$-vector is
  \[
  h_{\frac{e-1}2} = h_{\frac{e+1}2}=d_1d_2d_3 - \frac14
  \binom{d_1+d_2+d_3-d_4-1}3 + \frac14\binom{-d_1+d_2+d_3-d_4-1}3.
  \]

  When $e$ is even we want to compute the difference
  \begin{multline*}
    h_{\frac{e}2}-h_{\frac{e}2-1} = \sum_{i=0}^{d_4-1}
    h'_{\frac{e}2-i} - h'_{\frac{e-2}2-i} =
    h'_{\frac{e}2}-h'_{\frac{e-2}2-d_4+1} =
    h'_{\frac{e}2}-h'_{\frac{e}2-d_4} \\=
    h'_{\frac{d_1+d_2+d_3-3}2+\frac{d_4-1}2}-h'_{\frac{d_1+d_2+d_3-d_4}2-2}
    =
    h'_{\frac{d_1+d_2+d_3-3}2-\frac{d_4-1}2}-h'_{\frac{d_1+d_2+d_3-d_4}2-2}
    \\
    = h'_{\frac{d_1+d_2+d_3-d_4}2-1}-h'_{\frac{d_1+d_2+d_3-d_4}2-2}.
  \end{multline*}
  Again, we are in a range where we can use the expression
  $h'_j = \binom{j+2}{2}-\binom{j-d_1+2}{2}$ to conclude that
  \[
  h_{\frac{e}2}-h_{\frac{e}2-1} =
  \left[\frac{d_1+d_2+d_3-d_4}2\right]_+ -
  \left[\frac{-d_1+d_2+d_3-d_4}{2}\right]_+
  \]
  where $[x]_+=x$ for $x\ge 0$ and $[x]_+=0$ for $x<0$. For
  $d_4\le -d_1+d_2+d_3$ this expression equals $d_1$ and for
  $d\ge -d_1+d_2+d_3$ it equals the first term, which concludes the
  proof of the lemma.
\end{proof}

We denote by $Gor(H)$ the scheme parametrizing artinian Gorenstein
codimension 3 algebras $R/I$ with $h$-vector
$H=(1,h_1,\cdots , h_{e-1},h_e)$ \cite{Diesel}. We have

\begin{lemma}\label{dimGorH}
  Let $H=(1,h_1,\cdots , h_{e-1},h_e)$ and
  $H'=(1,h'_1,\cdots , h'_{e-1},h'_e)$ be the $h$-vectors of two
  artinian Gorenstein codimension 3 algebras with odd socle degree
  $e=2r+1$.  Assume that
  \[
  \begin{cases}
    h_i' = h_i, &\text{for $ i\ne r,r+1$,} \\
    h_i' = h_i-1,&\text{ for $i=r, r+1$.}
  \end{cases}
  \]
  Then, it holds:
  \[
  \dim Gor(H)-\dim Gor(H')=h_{r+1}-2h_{r+3}+h_{r+4}+1.
  \]
\end{lemma}
\begin{proof} By \cite[Example 5.2]{CV}, we have
  \[
  \dim Gor(H)=\frac12(3h_r+h_{r-1}-\sum _{i=0}^eh_ip_i)
  \]
  where $p_i=h_i-3h_{i-1}+3h_{i-2}-h_{i-3}.$ Therefore after a long
  but routine calculation we obtain
  \[
  \begin{array}{rcl} \dim Gor(H)-\dim Gor(H') & = & \frac12(3h_r+h_{r-1}-\sum _{i=0}^eh_ip_i)-\frac12(3h'_r+h'_{r-1}-\sum _{i=0}^eh'_ip'_i)\\
                                              & = &
                                                    h_{r+1}-2h_{r+3}+h_{r+4}+1
  \end{array}
  \]
  which proves what we want.
\end{proof}

\begin{theorem}\label{LI,n=4,d1d2d3d4}
  Let
  $I = \langle F_1, F_2 , F_3 ,F_4 \rangle \subset R =
  k[x_1,x_2,x_3,x_4]$
  be a complete intersection artinian ideal generated by general forms
  of degree $(d_1,d_2,d_3,d_4)$. Assume that
  $d_1\le d_2 \le d_3\le d_4$.  Let $e$ be the socle degree of $R/I$
  and let $(1,h_1,\cdots , h_{e-1},h_e)$ be the $h$-vector of $R/I$.
  Then
  \begin{enumerate}
  \item If $e$ is odd then the non-Lefschetz locus
    $\mathcal L_I\subset (\PP^{3})^\ast$ is a surface of degree
    \[
    h_{\frac{e-1}{2}}= d_1d_2d_3-\frac14 \binom{d_1+d_2+d_3-d_4-1}3 +
    \frac14\binom{-d_1+d_2+d_3-d_4-1}3.
    \]
  \item If $e$ is even then
    \[
    \codim \mathcal L_I = \min \{ h_{\frac{e}{2}}-h_{\frac{e}{2} -1}+1
    , 4 \}.
    \]
    In particular, $\mathcal L_I\subset (\PP^{n-1})^\ast$ is non-empty
    if and only if $h_{\frac{e}{2}}-h_{\frac{e}{2} -1}\le 2$ if and
    only if $d_4\ge d_1+d_2+d_3$ or
    $-d_1+d_2+d_3\le d_4\le d_1+d_2+d_3$ and $d_1+d_2+d_3-d_4\le 4$ or
    $d_4\le -d_1+d_2+d_3$ and $d_1\le 2$. In these cases
    $\delta _I:=\deg(\mathcal L_I)= \binom{ h_{\frac{e}{2}}}
    {h_{\frac{e}{2}}-h_{\frac{e}{2} -1}+1}$.
  \end{enumerate}

\end{theorem}
\begin{proof} Part (1) follows from Lemma \ref{HF,n=4,d1d2d3d4} taking
  into account that if $R/I$ has the WLP and the socle degree is odd
  then the non-Lefschetz locus is a surface of degree
  $h_{\frac{e-1}{2}}$.

  We now consider (2). By Remark \ref{large dn} (1) , if
  $d_4\ge d_3+d_2+d_1$ (remembering that $e$ is even, so
  $d_4 \neq d_3+d_2+d_1 -1$), then
  $h_{\frac{e}{2}}=h_{\frac{e}{2}-1}=d_1d_2d_3$ and
  $\mathcal L_I\subset (\PP^{n-1})^\ast$ is a surface of degree
  $d_1d_2d_3$. By Remark \ref{large dn} (2) , if $d_4= d_3+d_2+d_1-2$,
  then $h_{\frac{e}{2}}-h_{\frac{e}{2}-1}=1$,
  $h_{\frac{e}{2}}=d_1d_2d_3$ and $\mathcal L_I\subset (\PP^{3})^\ast$
  is an arithmetically Cohen-Macaulay curve of degree
  $\binom{d_1d_2\cdots d_{n-1}}{2}$.  By Proposition \ref{special case
    of conj} if $d_4=d_1+d_2+d_3-4$ then
  $\hbox{codim } {\mathcal L}_I= \min
  \{h_{\frac{e}{2}}-h_{\frac{e}{2}-1}+1,4\} =3$
  (where the last equality follows from Lemma \ref{HF,n=4,d1d2d3d4}
  (2)) and has degree $\binom{h_{\frac{e}{2}}}{3}$.

  From now on we assume $d_4 \leq d_1+d_2+d_3-6$. We fix a linear form
  $\ell $ and we set $S = R/(\ell )$.  We denote by
  $\mathcal A_{d_1,d_2,d_3,d_4}$ the set of almost complete
  intersection ideals $J \subset S$ of type $(d_1,d_2,d_3,d_4)$ and by
  ${\mathcal B}_{d_1,d_2,d_3,d_4}$ the set of almost complete
  intersection ideals $J\subset S$ of type $(d_1,d_2,d_3,d_4)$ and
  $h$-vector
  \[
  (1, h_1-1,h_2-h_1, \cdots , h_{\frac{e}{2}-1}-h_{\frac{e}{2}-2},
  h_{\frac{e}{2}}-h_{\frac{e}{2}-1}+1, 1).
  \]
  
  A general ideal $J$ in ${\mathcal A}_{d_1,d_2,d_3,d_4}$ can be
  linked by means of a complete intersection $K$ of type
  $(d_1,d_2,d_3)$ to a Gorenstein ideal $G$ with socle degree
  $s:=d_1+d_2+d_3-d_4-3$ and $h$-vector $H_G=(1,f_1,\cdots , f_s)$.  A
  general ideal $J'$ in ${\mathcal B}_{d_1,d_2,d_3,d_4}$ can be linked
  by means of a complete intersection $K$ of type $(d_1,d_2,d_3)$ to a
  Gorenstein ideal $G'$ with socle degree $s$ and $h$-vector
  $H_{G'}=(1,f'_1,\cdots , f'_s)$. Moreover, we have:
  \[
  \begin{array}{rcl} f_i' & = & f_i \text{ for } i\ne \frac{s-1}{2}, \frac{s+1}{2} \\
    f_i' & = & f_i-1 \text{ for } i=\frac{s-1}{2}, \frac{s+1}{2}
  \end{array}. 
  \]
  
  According to (\ref{to show}), to finish the proof it is enough to
  demonstrate that
  \[
  \dim Gor(H_G)-\dim Gor(H_{G'})=h_{\frac{e}{2}}-h_{\frac{e}{2}-1}+1
  \]
  (see also the end of the proof of Proposition \ref{LI,n=4} for the
  equivalence).
  \noindent Let us prove it. To this end, we denote by
  $(1,\tilde{h}_1,\cdots,\tilde{h}_w)$ the $h$-vector of the complete
  intersection ideal $K$ in $S$ of type $(d_1,d_2,d_3)$. So,
  $w=d_1+d_2+d_3-3$. Applying Lemma \ref{dimGorH}, we obtain
  \[
  \begin{array}{rcl}
    \dim Gor(H_G)-\dim Gor(H_{G'}) & = & f_{\frac{s-1}{2}+1}-2f_{\frac{s-1}{2}+3}+f_{\frac{s-1}{2}+4}+1 \\
                                   & = & (f_{\frac{s-1}{2}+1}-f_{\frac{s-1}{2}+2})+
                                         (f_{\frac{s-1}{2}+2}-2f_{\frac{s-1}{2}+3}+f_{\frac{s-1}{2}+4})+1\\
                                   & = & -\Delta f_{\frac{s-1}{2}+2} -\Delta ^2 f_{\frac{s-1}{2}+4}
                                         +1.
  \end{array}
  \]
  Since $G$ (resp. $G'$) is linked to $J$ (resp. $J'$) by a complete
  intersection $K$ of type $d_1, d_2, d_3$ we have
  $f_i=\tilde{h}_{i+d_4}$ for $i=\frac{s-1}{2}+1,
  \frac{s-1}{2}+2$. So, we get:
  \[
  \begin{array}{rcl} -\Delta f_{\frac{s-1}{2}+2} -\Delta ^2 f_{\frac{s-1}{2}+4} +1 & = & -\Delta \tilde{h}_{\frac{d_1+d_2+d_3+d_4}{2}}+\Delta ^2 \tilde{h}_{\frac{d_1+d_2+d_3+d_4}{2}+2} +1 \\
                                                                                   & = & \begin{cases} 1 & \text{ if } d_4\ge d_1+d_2+d_3 \\
                                                                                     \frac{d_1+d_2+d_3-d_4}{2}+1
                                                                                     &
                                                                                     \text{
                                                                                       if
                                                                                     }
                                                                                     -d_1+d_2+d_3\le
                                                                                     d_4\le
                                                                                     d_1+d_2+d_3
                                                                                     \\
                                                                                     d_1+1
                                                                                     &
                                                                                     \text{
                                                                                       if
                                                                                     }
                                                                                     d_4\le
                                                                                     -d_1+d_2+d_3.
                                                                                   \end{cases}
  \end{array}
  \]
  Therefore, applying Lemma \ref{HF,n=4,d1d2d3d4}, we conclude that
  \[
  \dim Gor(H_G)-\dim Gor(H_{G'})=h_{\frac{e}{2}}-h_{\frac{e}{2}-1}+1.
  \]
\end{proof}


\section{The non-Lefschetz locus of a general height three Gorenstein
  algebra}\label{sec:general Gorenstein}

When the Hilbert function is fixed, the height three Gorenstein
algebras with that Hilbert function lie in a flat family
\cite{Diesel}, so it makes sense to talk about the {\em general}
Gorenstein algebra in this family. From now on, we will abuse
terminology and refer to a general Gorenstein algebra, and assume that
it is understood that we have fixed the Hilbert function; we will also
assume that it is understood that in this section we refer {\em only}
to the height three situation, except for a small remark at the end of
the section.  In this section we will describe the codimension of the
non-Lefschetz locus of a general Gorenstein algebra, and in particular
describe exactly when it is of the expected codimension (given the
Hilbert function) in the sense of the earlier sections. One might
expect that just as with complete intersections, the general
Gorenstein algebra has non-Lefschetz locus of the expected
codimension, but this is not always the case. We give a classification
of those Hilbert functions for which the general Gorenstein algebras
fail to have non-Lefschetz locus of the expected codimension.

The Hilbert functions of height three Gorenstein algebras are
well-understood. They are the so-called {\em Stanley-Iarrobino (SI)}
sequences of height three. They are characterized as follows. A
sequence $\underline{h} = (1, 3, h_2,\dots,h_{e-1},h_e)$ is an
SI-sequence if and only if

\begin{itemize}
\item[(i)] $\underline{h}$ is symmetric.

\item[(ii)] Setting $g_i = h_i - h_{i-1}$ for
  $1 \leq i \leq \lfloor \frac{e}{2} \rfloor$, the sequence
  $\underline{g} = (1,2, g_2,\dots,g_{\lfloor \frac{e}{2} \rfloor})$
  satisfies Macaulay's growth condition.
\end{itemize}

\noindent Condition (ii) says that the sequence
$(1,3,h_2,\dots, h_{\lfloor \frac{e}{2} \rfloor})$ is the beginning of
the Hilbert function of some zero-dimensional scheme in $\mathbb P^2$
of degree $h_{\lfloor \frac{e}{2} \rfloor}$. It is important to note
that it does not mean that for every Gorenstein algebra $R/I$ with
this Hilbert function, the components of $R/I$ up to degree
$\lfloor \frac{e}{2} \rfloor$ actually coincide with the corresponding
components of a zero-dimensional scheme. If such a condition does
hold, and if the zero-dimensional scheme is reduced, we will say that
$R/I$ ``comes from points." For any SI-sequence, by taking a suitable
Gorenstein quotient of the coordinate ring of a suitable reduced set
of points, there is always a subfamily (of the Gorenstein family
corresponding to the SI-sequence) that {\em does} come from points.

We say that a sequence $(1,2, g_2, g_3, \dots, g_k)$ is of {\em
  decreasing type} if begins with $(1,2,3,\dots)$ (growing with the
polynomial ring $k[x,y]$), then is possibly flat, then is strictly
decreasing.

\begin{theorem} \label{gormain} Fix an SI-sequence
  $\underline{h} = (1, 3, h_2, \dots, h_{e-2}, 3,1)$ of socle degree
  $e$.

  \begin{itemize}
  \item[(i)] If there are two or more consecutive values of $h_i$ that
    are equal then the general Gorenstein algebra with Hilbert
    function $\underline{h}$ has non-Lefschetz locus of the expected
    codimension, namely one. This holds, in particular, when $e$ is
    odd.

  \item[(ii)] Assume $e$ is even and
    \begin{equation} \label{si seq} \underline{h} = (1,3,h_2,\dots,
      h_{ \frac{e}{2}-1}, h_{\frac{e}{2}}, h_{\frac{e}{2} +1}, \dots,
      h_{e-2}, 3, 1)
    \end{equation}
    where $h_{\frac{e}{2} -1} < h_{\frac{e}{2}} > h_{\frac{e}{2} +1}$.
    Let $\underline{g}$ be the sequence of positive first differences,
    as above.  Then the general Gorenstein algebra with this Hilbert
    function has non-Lefschetz locus of the expected codimension if
    and only if $\underline{g}$ is of decreasing type. If
    $\underline{g}$ is not of decreasing type then the non-Lefschetz
    locus has codimension one.

  \end{itemize}
\end{theorem}

\begin{proof}
  Suppose
  $h_{\lfloor \frac{e}{2} \rfloor} = h_{\lfloor \frac{e}{2} \rfloor
    +1}$,
  for instance if the socle degree $e$ is odd. Then the expected
  codimension of the non-Lefschetz locus is one. Since it is known
  that the general height three artinian Gorenstein algebra with any
  given Hilbert function has the WLP \cite{Harima}, the non-Lefschetz
  locus of the general Gorenstein algebra with odd socle degree has
  the expected codimension.  So from now on assume that $e$ is even,
  and that
  $h_{\frac{e}{2} -1} < h_{\frac{e}{2}} > h_{\frac{e}{2} +1}$.

  We first assume that $\underline{g}$ is of decreasing type.  Our
  strategy will be to construct an explicit Gorenstein algebra having
  such a Hilbert function and non-Lefschetz locus of expected
  codimension; then by semicontinuity the general Gorenstein algebra
  with this Hilbert function has non-Lefschetz locus of the expected
  codimension.

  So consider the SI-sequence (\ref{si seq}), and assume that its
  first difference is of decreasing type. Let $Z$ be a reduced set of
  $h_{\frac{e}{2}}$ points in $\mathbb P^2$ with Hilbert function
  given by
  \[
  (1,3, h_2,\dots,h_{\frac{e}{2} -1}, h_{\frac{e}{2}},
  h_{\frac{e}{2}}, \dots).
  \]
  The $h$-vector of $Z$ is given by the first difference sequence
  $\underline{g}$. Let $I$ be an artinian Gorenstein ideal obtained as
  a suitable quotient of $R/I_Z$, so that the Hilbert function of
  $R/I$ is precisely $\underline{h}$. (See \cite{BoijPoints}.)  This
  means that $[I]_i = [I_Z]_i$ for $i \leq \frac{e}{2}$. We want to
  show:

  \begin{itemize}

  \item[(i)] $\mathcal L_I = \emptyset$ if $g_{\frac{e}{2}} \geq 2$;

  \item[(ii)] $\hbox{codim } \mathcal L_I = 2$ if
    $g_{\frac{e}{2}} = 1$;

  \end{itemize}

  \noindent We have already seen that $\hbox{codim } \mathcal L_i = 1$
  if $g_{\frac{e}{2}} = 0$, so we have assumed
  $h_{\frac{e}{2}-1} < h_{\frac{e}{2}}$.  Because $\underline{g}$ is
  of decreasing type, we can assume that $Z$ has the Uniform Position
  Property (UPP) by a result by Maggioni and
  Ragusa~\cite{MaggioniRagusa}.  In particular it has the
  2-Cayley-Bacharach Property: the Hilbert functions of $Z$ minus a
  point are all the same, and the Hilbert functions of $Z$ minus two
  points are all the same. We consider the multiplication on $R/I_Z$
  from degree $\frac{e}{2}-1$ to degree $\frac{e}{2}$ by a linear form
  $\ell$. Notice that by UPP, $\ell$ vanishes on at most two points
  since $h_1 = 3$ (so not all points lie on a line).

  \medskip

  \noindent \underline{Case 1}: {\em $\ell$ does not vanish on any
    point of $Z$.}

  \smallskip

  Then $\ell$ is a non-zerodivisor, so the multiplication is injective
  and $\ell$ is a Lefschetz element.

  \medskip

  \noindent \underline{Case 2}: {\em $\ell$ vanishes at exactly one
    point, $P$, of $Z$.}

  \smallskip

  Let $Y = Z \backslash P$, defined by $I_Y = I_Z : \ell$. Notice that
  since $h_{\frac{e}{2}-1} < h_{\frac{e}{2}}$, and $Z$ has the UPP, we
  have $[I_Y]_{\frac{e}{2}-1} = [I_Z]_{\frac{e}{2}-1}$. From the
  diagram
  \[
  \begin{array}{ccccccccccc}
    &&&& \displaystyle [R/I]_{\frac{e}{2}-1} & \displaystyle \stackrel{\times \ell}{\longrightarrow} & \displaystyle [R/I]_{\frac{e}{2}} \\
    &&&& || \ \ \ && || \ \\
    0 & \rightarrow &\displaystyle [I_Y / I_Z]_{{\frac{e}{2} -1}} & \rightarrow & \displaystyle [R/I_Z]_{\frac{e}{2}-1} & \displaystyle \stackrel{\times \ell}{\longrightarrow} & \displaystyle [R/I_Z]_{\frac{e}{2}} \\
    && || \\
    && 0
  \end{array}
  \]
  we see that multiplication by $\ell$ is again injective, i.e. $\ell$
  is a Lefschetz element.

  \medskip

  \noindent \underline{Case 3}: {\em $\ell$ vanishes at exactly two
    points, $P$ and $Q$, of $Z$.}

  \smallskip

  We obtain the same diagram as in Case 2. In this case, though, we
  have $[I_Y/I_Z]_{\frac{e}{2}-1} = 0$ if and only
  $g_{\frac{e}{2}} \geq 2$. If $g_{\frac{e}{2}} = 1$, then
  $[I_Y/I_Z]_{{\frac{e}{2}-1}} \neq 0$. Thus
  $\mathcal L_I = \emptyset$ if $g_{\frac{e}{2}} \geq 2$, and
  $\hbox{codim } \mathcal L_I = 2$ if $g_{\frac{e}{2}} = 1$, both of
  which correspond to the expected codimension. In the latter case,
  the degree formula gives
  $\deg \mathcal L_I = \binom{h_{\frac{e}{2}}}{2}$, which can be seen
  directly as all choices of two points of $Z$.

  Thus we have constructed an explicit Gorenstein algebra with Hilbert
  function $\underline{h}$ and non-Lefschetz locus of the expected
  codimension, so as noted above, by semicontinuity the general
  Gorenstein algebra has non-Lefschetz locus of the expected
  codimension.

  It remains to consider the case where $\underline{g}$ is not of
  decreasing type. In this case the same approach will not work, since
  it is a priori possible that the Gorenstein algebras coming from
  points in this case fail to have non-Lefschetz locus of the expected
  codimension, but nevertheless the general one does. We will show by
  a different method that this is not the case.

  So assume that $g_{i-1}= g_{i}$ for some $i \leq \frac{e}{2}$, and
  that $i-1$ is the least degree for which $g_{i-2} > g_{i-1} = g_i$.
  Assume also that $R/I$ is general in the flat family of Gorenstein
  algebras with this Hilbert function. By a result of Ragusa and
  Zappal\`a~\cite{RZ2} the generators of $I$ of degree $\leq i-1$ all
  have a common factor, $F$, of degree $g_i$. Furthermore, the
  generators of the ideal $I : F$ of degree $\leq i-1$ span the ideal
  of a reduced set of points, $Z$, in $\mathbb P^2$.

  In order to prove our statement on the non-Lefschetz locus, let
  $\ell$ be a linear form and let $Y$ be defined by $I_Z : \ell$. This
  time we consider the multiplication from degree $i-1$ to degree $i$.
  We have a diagram
  \[
  \begin{array}{ccccccccccc}
    &&&& \displaystyle [R/I]_{i-1} & \displaystyle \stackrel{\times \ell}{\longrightarrow} & \displaystyle [R/I]_{i} \\
    &&&& || \ \ \ && || \ \\
    0 & \rightarrow &\displaystyle [I_Y \cdot F / I_Z \cdot F]_{{i-1}} & \rightarrow & \displaystyle [R/I_Z \cdot F]_{i-1} & \displaystyle \stackrel{\times \ell}{\longrightarrow} & \displaystyle [R/I_Z \cdot F]_{i} \\
  \end{array}
  \]
  But $\dim_k [R/I_Z]_j$ reaches its multiplicity in degree $i-2-g_i$
  \cite{davis}, so whenever $\ell$ vanishes at a point of $Z$,
  $[I_Y \cdot F / I_Z \cdot F]_{i-1}$ is not zero and the
  multiplication fails to have maximal rank. Thus the non-Lefschetz
  locus in degree $i-1$ has codimension one. By Proposition
  \ref{inclusions} we are done.
\end{proof}

\begin{remark}
  Using the ideas from the proof of Theorem \ref{gormain}, it is easy
  to construct an artinian Gorenstein algebra whose non-Lefschetz
  locus is non-reduced, for almost any SI-sequence $\underline{h}$. We
  simply relax the generality condition on $Z$ and allow three points
  to lie on a line. The only obstacle is when the $h$-vector of $Z$
  does not allow this, i.e. when it is $(1)$, $(1,1)$
\end{remark}

\begin{remark}
  If $\underline{g}$ is not of decreasing type and $A$ has odd socle
  degree, then even though $\mathcal L_I$ is of the expected
  codimension, it is still true that its behavior is not the expected
  one because in earlier degrees it is a hypersurface when we expect
  it not to be.
\end{remark}

\begin{remark}
  Since complete intersections of codimension three have $g$-vectors
  of decreasing type, Theorem~\ref{gormain} is a generalization of
  Theorem~\ref{LI,n=3,d1d2d3}. However, since the method of proof is
  completely different, we prefer to give the proof of
  Theorem~\ref{LI,n=3,d1d2d3} in Section~\ref{sec:general ci}.
\end{remark}

\section{The non-Lefschetz locus in codimension two}\label{sec:codim2}

In this short section we describe the situation in codimension two.
Let $R = k[x,y]$ and let $I$ be an artinian ideal in $R$.  Now the
Hilbert function of $R/I$ has the form $(1,2, h_2, \dots, h_e )$,
where $h_i = i+1$ until the initial degree of $h_{R/I}$, and then is
non-decreasing from then on.  Furthermore, if $h_i = h_{i+1}$ for some
$i$, this represents maximal growth of the Hilbert function, so
Macaulay's theorem \cite{Macaulay} together with Gotzmann's theorem
\cite{Gotzmann} gives that the greatest common divisor of all the
elements in $I$ of degree $i$ and degree $i+1$ has degree $h_i$.

For fixed Hilbert function, the algebras having that Hilbert function
form an irreducible family.  For a general such algebra, if
$h_{i+1} < h_i$ then the elements of $I$ in degree $i+1$ do not have a
common divisor.

\begin{lemma}
  Let $I$ be any artinian graded ideal in $R = k[x,y]$.  Let
  $\{ h_i \}$ be the $h$-vector of $R/I$.  Fix any degree $i$.  There
  exists a linear form $\ell $ such that
  $\times \ell : [R/I]_{i-1} \rightarrow [R/I]_{i}$ fails to have
  maximal rank if and only if $I$ has a common factor, say $F$,
  between all forms of degree~$i$.  We have $\deg F \leq h_i$.
\end{lemma}

\begin{proof}
  The fact that the degree of a GCD in degree $i$ must be $\leq h_i$
  is well known.  Assume that the forms of degree $i$ in $I$ have a
  GCD, say $F$, of positive degree.  If $F \in I$ then $[I]_i$ is the
  degree $i$ part of a principal ideal $(F)$, we have
  $h_{i-1} = h_i = \deg F$.  But in $R$, $F$ factors into linear
  factors.  Thus clearly the non-Lefschetz locus consists precisely of
  the factors of $F$ (counted with multiplicity).  That is, the locus
  in $(\mathbb P^1)^\ast$ of linear forms $\ell $ for which
  $\times \ell : [R/I]_{i-1} \rightarrow [R/I]_i$ fails to have
  maximal rank is zero-dimensional of degree equal to $\deg F = h_i$.
  (This is not quite the same as the non-Lefschetz locus since we are
  looking only in degrees $i-1$ and $i$.)

  Suppose instead that the GCD, $F$, is not in $I$ and has degree
  $d \leq h_i$.  We have $\dim [I]_i = i+1-h_i = m$, say. Choose a
  basis for $[I]_i$ of the form $\{ FA_1,\dots, FA_{m} \}$.  Say $F$
  factors as $F = \ell _1 \cdots \ell _d$.  For each factor of $F$,
  for instance $\ell _1$, we have $m$ independent elements of
  $[R]_{i-1}$ such that multiplication by $\ell _1$ is zero in $R/I$.
  Now,
  \[
  h_{i-1} - h_i = (i - \dim [I]_{i-1}) - (i+1 - \dim [I]_i ) = m-1 -
  \dim[I]_{i-1} < m
  \]
  so multiplication by $\ell _1$ has a larger kernel than expected
  (surjectivity implies a kernel of dimension $h_{i-1} - h_i$) and so
  fails to have maximal rank.

  Conversely, assume that $I$ does not have a GCD in degree $i$.  We
  want to show that multiplication by any linear form $\ell $ gives a
  homomorphism of maximal rank from degree $i-1$ to degree $i$.  In
  degrees smaller than the initial degree of $I$, $R/I$ agrees with
  the polynomial ring, so the result is clear.  If $h_{i-1} = h_i$
  then by the result of Davis~\cite{davis} $I$ has a GCD in degree
  $i$.  Thus we may assume that $h_{i-1} > h_i$.  Suppose that there
  exists a linear form $\ell $ for which the corresponding
  multiplication from degree $i-1$ to degree $i$ is not surjective.
  Consider the exact sequence
  \[
  0 \rightarrow [R/(I:\ell ) (-1)]_i \stackrel{\times \ell
  }{\longrightarrow} [R/I]_i \rightarrow [R/(I,\ell )]_i \rightarrow
  0.
  \]
  By assumption, $[R/(I,\ell )]_i \neq 0$.  But
  $R/(\ell ) \cong k[x]$.  This means that the restriction of $[I]_i$
  modulo $\ell $ is zero.  This can only happen if $\ell $ is a GCD
  for $[I]_i$, contradicting our assumption.  The result follows.
\end{proof}

\begin{proposition}
  Let $R = k[x,y]$.  Fix a Hilbert function $\{ h_i \}$ that exists
  for artinian graded quotients of $R$.  Let $R/I$ be a general
  algebra with this Hilbert function.  For any $i$, there exists a
  linear form $\ell $ such that
  $\times \ell : [R/I]_{i-1} \rightarrow [R/I]_i$ fails to have
  maximal rank if and only if $h_{i-1} = h_i$.  In particular, if
  $R/I$ is a general complete intersection of type $(d_1,d_2)$, with
  $d_1 \leq d_2$, then the non-Lefschetz locus is empty if and only if
  $d_1 = d_2$.  Otherwise, the degree of the non-Lefschetz locus is
  $d_1$.
\end{proposition}


\end{document}